%% file: main_lccvstb2.tex
\documentclass[11pt,a4paper,reqno]{amsart}
\usepackage{amsfonts}
\usepackage{amsmath}
\usepackage{amssymb}
\usepackage{graphicx}

\usepackage{tikz}

\usepackage[margin=20mm]{geometry}

\usepackage{multirow}
\usepackage{array}

\usepackage[colorlinks, linkcolor=blue,  citecolor=blue, urlcolor=blue, bookmarks=false]{hyperref}%
\hypersetup{pdfstartview=FitH}
\usepackage{url}
\usepackage{xspace}

\usepackage{dsfont}

\usepackage{verbatim}
\usepackage{alltt}
\usepackage{graphicx}
\graphicspath{{../figures/}{../figuresOld/}}

\newtheorem{theorem}{Theorem}[section]

\newtheorem{lemma}[theorem]{Lemma}

\newtheorem{proposition}[theorem]{Proposition}
\newtheorem{remark}[theorem]{Remark}

\numberwithin{equation}{section}

\def\XXint#1#2#3{{\setbox0=\hbox{$#1{#2#3}{\int}$}
     \vcenter{\hbox{$#2#3$}}\kern-.5\wd0}}

\let\originalleft\left
\let\originalright\right
\renewcommand{\left}{\mathopen{}\mathclose\bgroup\originalleft}
\renewcommand{\right}{\aftergroup\egroup\originalright}

\renewcommand{\ref}[1]{\autoref{#1}}

\AtBeginDocument{
}

\usepackage{enumitem}
\setlist[description]{align=right, labelwidth=2cm, leftmargin=!, font=\normalfont\bfseries}

\begin{document}
\title[Sharp stability for log-concave sources]{Sharp potential and Brenier map $W_2$-stability for log-concave sources}
\author{Yuanlong Ruan}
\thanks{Copyright \copyright~2026 The Author. All rights reserved.}
\address{Beihang University, Beijing, China}


\date{\today}

\begin{abstract}
We study quantitative stability of Brenier potentials and maps under changes of the
target measure. Let $\kappa>0$, for a $\kappa$-strongly log-concave source $\rho$,
possibly with a non-smooth potential, the normalized transport potentials satisfy
\[
\left\Vert u_{\mu}-u_{\nu}\right\Vert _{L^{2}(\rho)}\leqslant\left(
1/\kappa\right)  ^{1/2}W_{2}\left(  \mu,\nu\right)
\]
for arbitrary targets with finite second moments. This readily yields a Poincaré inequality for $\kappa$-strongly log-concave sources. The coefficient $(1/\kappa)^{1/2}$ is shown to be optimal. The proof relies on a probabilistic interpretation of cell energies and the subsequent variance decomposition. Interpolation from potential to map then yields the map stability for targets in a common compact set,
\[
\left\Vert T_{\mu}-T_{\nu}\right\Vert _{L^{2}(\rho)}\leqslant \operatorname{Const} \cdot W_{2}\left(  \mu,\nu\right)  ^{1/3}.
\]
Sharpness of the order for both the potential and map stability are established. This clarifies a long-
standing limitation in topological embedding: representing distributions by their
Brenier maps preserves convergence, but it need not preserve local distances.

\end{abstract}

\maketitle

\input{lccvmain2}

\bibliographystyle{plain}
\bibliography{main}
\end{document}

%% file: lccvmain2.tex
\section{Introduction and main results}

\label{sec:introduction}

Quantitative $W_{1}$-stability of Brenier maps has been studied in a series of
works \cite{berman2021convergence}\cite{delalande2023quantitative}
\cite{merigot2026sharp}\cite{letrouit2026gluing}. Recent improvements on the
$W_{1}$-stabilities are obtained in \cite{ruan2026sharp} under relaxed source
and target settings. When both source and targets are compactly supported, the
stability is well-studied and the $\left(  1/4\right)  $-exponent is shown to
be sharp for the $W_{1}$-stability. This is independently obtained by
\cite{merigot2026sharp}\cite{ford2026quantitative}\cite{ruan2026sharp}
\cite{wang2026quantitative} with sharpness verified by the extreme case
construction \cite{letrouit2026unstable}. Underlying these works are the
geometries of convex duality \cite{carlier2023fenchel}
\cite{carlier2025quantitative}.

When the source is uniform in a compact convex set and target in a compact
set, the $W_{2}$-stability of Brenier maps is proved by
\cite{merigot2020quantitative} with exponent $2/15.$
\cite{wang2026quantitative} also studies $W_{2}$-stability under more general
source target assumptions and $\left(  1/4\right)  $-exponent is shown to be
sharp for the class of bounded variation sources with finite second moment and
bounded targets.

In a recent preprint, \cite[Proposition 5.1]{openai2026sharp} proved the
linear potential $W_{2}$-stability under the same setting as
\cite{merigot2020quantitative}. The sharp one-third exponent of Brenier map
$W_{2}$-stability was obtained by a sharp interpolation from potential
stability. The main tool used there is the finite cell calculation usually
seen in semi-discrete optimal transport. The result there also alleviated
concern that potential to map conversion might cause loss of information.

The current article establishes the sharp one-third exponent for Brenier map
$W_{2}$-stability with log-concave sources. The cell calculus is used as in
\cite{openai2026sharp}, but the proofs of the two deviate at the central part
and the estimates needed to obtain the sharp coefficient $1/\kappa$ for the
log-concave sources are not available in \cite{openai2026sharp}. For a compact
source, all cell locations are uniformly bounded. In our case, the cells and
their faces are unbounded, and the integrands contain the density $f$, $xf$,
and quadratic coordinate factors. Unlike the uniform source, the new estimate
used here is for the \textit{normalized cell barycenters }$b_{i}=B_{i}/M_{i}$
which can be arbitrarily far from the area where large source probability
concentrates. A bound on raw moments of these quantities alone is not enough.
The need to obtain the sharp estimate makes the difficulty even more
prominent. To overcome the difficulty and obtain the sharp coefficient, we
interpret the cell dynamic energy as total variance of the cells in Lemma
\ref{lem:joint}, the subsequent variance decomposition then allows us to bound
the squared displacement with a sharp coefficient $1/\kappa$ independent of
the source support radius, and eventually establish the potential stability in
Proposition \ref{prop:finite-strong}.

Finally, the Brenier map stability with sharp exponent $1/3$ is obtained
through a sharp potential to gradient interpolation given in Section
\ref{sec:BV}. For compact domain with rectifiable boundary, the conversion is
obtained in \cite[Proposition 4.1]{delalande2023quantitative}. It immdediately
follows that, up to constant coefficients,
\[
W_{2}\left(  \mu,\nu\right)  \leqslant\left\Vert T_{\mu}-T_{\nu}\right\Vert
_{L^{2}(\rho)}\leqslant W_{2}\left(  \mu,\nu\right)  ^{1/3}.
\]
This not only shows that convergence in $W_{2}$-distance is equivalent to
convergence in the transport map representation for compact targets, but also
clarifies a long-standing numerical limitation: representing log-concave
distributions by their Brenier maps preserves convergence, but it need not
preserve local distances.

At the time of updating an old arxiv version, a related preprint \cite{letrouit2026sharp} independently proves an equivalent result to Theorem \ref{thm:strong-main}. Their method is based on linearization of the Monge-Ampère equation, whereas we use cell calculus and cell variance decomposition. The sharp constant $1/\kappa$ is obtained in both works. Our results also include sharp Brenier map stability on $\mathbb{R}^d$. The stabilities are extended to general isotropic log-concave sources \cite{wang2026one} in weighted form.

\subsection{Notations}

$B\left(  0,R\right)  $ denotes an open ball with radius $R>0.$

$BV\left(  \mathbb{R}^{d}\right)  $ is the space of functions of bounded
variations \cite[Chapter 3]{ambrosio2000functions}.

Let $\epsilon>0,$ a mollifier $\zeta_{\epsilon}\in C_{c}^{\infty}
(\mathbb{R}^{d})$ is a nonnegative smooth function supported in $\bar
{B}\left(  0,\epsilon\right)  $ with integration equal to $1.$ And we always
use \textbf{symmetric} mollifier.

For a probability $\rho$ and $g\in L^{2}(\rho)$, $\operatorname{Var}_{\rho
}\left(  g\right)  =\int g^{2}d\rho-\left(  \int gd\rho\right)  ^{2}$.

The srandard Gaussian distribution in $\mathbb{R}^{d}$ is denoted by
$\gamma_{d}.$

\subsection{Main results}

Throughout we assume that the source measure is log-concave as defined below.
Let $V:\mathbb{R}^{d}\longmapsto\mathbb{R\cup}\left\{  +\infty\right\}  $ be
proper lower semicontinuous convex,
\begin{equation}
d\rho=fdx=Z^{-1}e^{-V\left(  x\right)  }dx\text{ and }0<Z=\int_{\mathbb{R}
^{d}}e^{-V\left(  x\right)  }dx<\infty, \label{eq:logconcave-source}
\end{equation}
Here $V$ may take $\infty,$ but that $Z>0$ forces the effective domain of $V$
to have full dimension and non-empty interior. Denote the interior of the
effective domain by
\begin{equation}
U=\operatorname*{int}\left\{  x\in\mathbb{R}^{d}:V\left(  x\right)
<\infty\right\}  . \label{eq:V-effective-domain}
\end{equation}
We also assume that $V$ is $\kappa$-strongly convex
\begin{equation}
V\left(  x\right)  -\tfrac{\kappa}{2}\left\vert x\right\vert ^{2}\text{ is
convex for some }\kappa>0. \label{eq:strong-source}
\end{equation}
This is equivalent to
\begin{equation}
V\left(  y\right)  \geqslant V\left(  x\right)  +p\cdot\left(  y-x\right)
+\frac{\kappa}{2}\left\vert y-x\right\vert ^{2},\text{ }\forall x,\text{
}y\text{ and }p\in\partial V\left(  x\right)  . \label{eq:strong-source1}
\end{equation}
The assumption guarantees that $V$ is coercive and has a unique minimizer
$x_{\ast}$. Since $0\in\partial V(x_{\ast})$, $\left(  \ref{eq:strong-source1}
\right)  $ indicates that $V\left(  x\right)  \geqslant V\left(  x_{\ast
}\right)  +\kappa\left\vert x-x_{\ast}\right\vert ^{2}/2$. Thus
\begin{equation}
0<f\left(  x\right)  \leqslant Z^{-1}e^{-V\left(  x_{\ast}\right)  }
e^{-\kappa\left\vert x-x_{\ast}\right\vert ^{2}/2}.
\label{eq:Gaussian-envelope}
\end{equation}

\begin{theorem}
[Potential and map stability]\label{thm:strong-main}Assume that $\left(
\ref{eq:strong-source}\right)  $ holds. Then the following hold.

$\left(  i\right)  $ For any $\mu\in\mathcal{P}_{2}(\mathbb{R}^{d})$ there is
a unique finite convex function $u_{\mu}:\mathbb{R}^{d}\mapsto\mathbb{R}$
satisfying $u_{\mu}\in L^{2}\left(  \rho\right)  $,
\[
\left(  \nabla u_{\mu}\right)  _{\#}\rho=\mu\text{,}
\]
and is normalized to have zero mean $\int u_{\mu}d\rho=0.$ Its gradient is the
quadratic Brenier map $T_{\mu}$.

$\left(  ii\right)  $ For all $\mu,\nu\in\mathcal{P}_{2}(\mathbb{R}^{d})$,
their corresponding normalized potential $u_{\mu},u_{\nu}$ satisfy,
\begin{equation}
\left\Vert u_{\mu}-u_{\nu}\right\Vert _{L^{2}(\rho)}\leqslant\left(
1/\kappa\right)  ^{1/2}W_{2}\left(  \mu,\nu\right)  .
\label{eq:strong-potential}
\end{equation}

$\left(  iii\right)  $ The density $f$ belongs to $BV(\mathbb{R}^{d})$. If
both targets are supported in a common closed ball of radius $L>0$, then, with
$B_{f}=\left\vert Df\right\vert \left(  \mathbb{R}^{d}\right)  ,$
\begin{equation}
\left\Vert T_{\mu}-T_{\nu}\right\Vert _{L^{2}(\rho)}\leqslant2\sqrt{2}\left(
1/\kappa\right)  ^{1/6}L^{2/3}B_{f}^{1/3}W_{2}\left(  \mu,\nu\right)  ^{1/3}.
\label{eq:strong-map}
\end{equation}

\end{theorem}

Note that, in an equivalent form, the stability is $\left(
\ref{eq:strong-potential}\right)  $ indeed a Poincar\'{e} inequality for
$\kappa$-strongly log-concave source,
\[
\operatorname{Var}_{\rho}\left(  u_{\mu}-u_{\nu}\right)  \leqslant
1/\kappa\left\Vert \nabla u_{\mu}-\nabla u_{\nu}\right\Vert _{L^{2}(\rho)}
^{2}.
\]

\begin{theorem}
[Sharpness at the fixed Gaussian]\label{thm:Gaussian-main}Let $d\geqslant2$
and let $\rho=\gamma_{d}$ be the standard Gaussian probability on
$\mathbb{R}^{d}$. There are targets $\mu_{h},\nu_{h}$ supported in the closed
unit ball, for all sufficiently small $h>0$, and positive constants
$\kappa_{0},\kappa_{1},\kappa_{2}$ such that, as $h\rightarrow0,$
\[
W_{2}\left(  \mu_{h},\nu_{n}\right)  \sim\kappa_{0}h^{3/2},
\]
\[
\left\Vert T_{\mu}-T_{\nu}\right\Vert _{L^{2}(\gamma_{d})}\sim\kappa
_{1}h^{1/2},
\]
\[
\left\Vert u_{\mu_{h}}-u_{\nu_{h}}\right\Vert _{L^{2}\left(  \gamma
_{d}\right)  }\sim\kappa_{3}h^{3/2}.
\]
Here $u_{\mu_{h}},u_{\nu_{h}}$ are the zero mean potentials of
Theorem~\ref{thm:strong-main}, which applies with $\kappa=1$. Consequently no
map exponent larger than $1/3$ can replace $1/3$ in $\left(
\ref{eq:strong-map}\right)  $. The same family also realizes linear order
potential stability and $1/3$ conversion from potentials to maps.
\end{theorem}

\section{Calculations on weighted cells}

Fix $N\in\mathbb{N}$, \textbf{distinct} sites $y=\left(  y_{1},\ldots
,y_{N}\right)  \in\left(  \mathbb{R}^{d}\right)  ^{N}$ and intercepts
$c=\left(  c_{1},\ldots,c_{N}\right)  \in\mathbb{R}^{N}$. Let
\begin{equation}
u_{c,y}\left(  x\right)  =\max_{1\leqslant i\leqslant N}\left\{  x\cdot
y_{i}+c_{i}\right\}  . \label{eq:def-u-cy}
\end{equation}
It defines the associated power cells. On the region where the $i$-th linear
function strictly dominates, we have $\nabla u_{c,y}=y_{i}$. Thus source mass
in that region is sent to $y_{i}$, and by moving the boundaries the intercepts
$c_{i}$ adjust how much mass is sent there. The cells form a partition of
$\mathbb{R}^{d}$ with the $i$-th being
\begin{equation}
P_{i}\left(  c,y\right)  =\left\{  x\in\mathbb{R}^{d}:x\cdot y_{i}
+c_{i}\geqslant x\cdot y_{j}+c_{j},\text{ }\forall j\right\}  .
\label{eq:Def-Pi}
\end{equation}
These cells are convex polyhedra. See \cite[Remark 17]{merigot2021optimal} for
a Laguerre description. The source $\rho=fdx$ allocates weight to each cell by
how much mass the cell contains
\begin{equation}
M_{i}\left(  c,y\right)  =\int_{P_{i}\left(  c,y\right)  }f\left(  x\right)
dx. \label{eq:def-Mi}
\end{equation}
The barycenter of each cell is thus
\begin{equation}
b_{i}\left(  c,y\right)  =\frac{B_{i}\left(  c,y\right)  }{M_{i}\left(
c,y\right)  }\text{ with }B_{i}\left(  c,y\right)  =\int_{P_{i}\left(
c,y\right)  }xf\left(  x\right)  dx. \label{eq:def-Bi}
\end{equation}
The conditional variance is defined by
\begin{equation}
\Sigma_{i}\left(  c,y\right)  =\frac{1}{M_{i}\left(  c,y\right)  }\int
_{P_{i}\left(  c,y\right)  }\left[  x-b_{i}\left(  c,y\right)  \right]
\left[  x-b_{i}\left(  c,y\right)  \right]  ^{\top}f\left(  x\right)  dx.
\label{eq:def-Sigma}
\end{equation}
Our convention is that vectors are columnwise. Here $\top$ denotes
transposition. When the dependence on $\left(  c,y\right)  $ has no ambiguity
from the context, we will omit it and write
\begin{equation}
P_{i}=P_{i}\left(  c,y\right)  ,\text{ }M_{i}=M_{i}\left(  c,y\right)  ,\text{
}B_{i}=B_{i}\left(  c,y\right)  ,\text{ }\Sigma_{i}=\Sigma_{i}\left(
c,y\right)  . \label{eq:short-Pi-Mi-Bi}
\end{equation}
For the above definitions to make sense, the cell must have positive mass
$M_{i}\left(  c,y\right)  >0.$ So the collection of sites and intercepts we
work with are
\[
\mathcal{Y}_{N}=\left\{  y\in\left(  \mathbb{R}^{d}\right)  ^{N}:y_{i}\neq
y_{j}\text{ if }i\neq j\right\}  ,
\]
\[
\Omega_{N}=\left\{  \left(  c,y\right)  \in\mathbb{R}^{N}\times\mathcal{Y}
_{N}:M_{i}\left(  c,y\right)  >0\text{ for all }i\right\}  .
\]
The requirement of $\Omega_{N}$ does not excluding the possibiliy of the site
locations $y_{i}$ lying outside $U,$ the interior of the effective domain$.$

The first thing we need to consider is how the cell mass and barycenter
respond to parameter pertubrations, this is the topic of the following lemma.

\begin{lemma}
\label{lem:weighted-calculus}Assume that $\left(  \ref{eq:strong-source}
\right)  $ holds. Fix $N\in\mathbb{N}$ and $\left(  c,y\right)  \in
\mathbb{R}^{N}\times\mathcal{Y}_{N}$. Then the following hold.

(i) The set $\Omega_{N}$ is open and $M_{i},$ $B_{i}\in C^{1}\left(
\Omega_{N}\right)  $ for all $i.$

(ii) For fixed $(c,y)\in\Omega_{N}$, let $i,j\in\left\{  1,\ldots,N\right\}
$. Then for $q\in\mathbb{R}^{N}$, $v\in(\mathbb{R}^{d})^{N}$, the directional
derivatives are given by
\begin{equation}
DM_{i}\left[  q,v\right]  =\sum_{j\neq i}\int_{F_{ij}}\frac{\left(
q_{i}-q_{j}\right)  +x\cdot(v_{i}-v_{j})}{\left\vert y_{i}-y_{j}\right\vert
}f\left(  x\right)  d\mathcal{H}^{d-1}\left(  x\right)  ,
\label{eq:mass-derivative}
\end{equation}
\begin{equation}
DB_{i}\left[  q,v\right]  =\sum_{j\neq i}\int_{F_{ij}}\frac{\left(
q_{i}-q_{j}\right)  +x\cdot(v_{i}-v_{j})}{\left\vert y_{i}-y_{j}\right\vert
}xf\left(  x\right)  d\mathcal{H}^{d-1}\left(  x\right)  .
\label{eq:mm-derivative}
\end{equation}

(iii) The linear map $G:\mathbb{R}^{N}\mapsto\mathbb{R}^{N}$ defined by
$Gq=\left(  DM_{i}\left[  q,0\right]  \right)  _{i=1}^{N}$ has kernel
$\operatorname{span}\left\{  \left(  1,\ldots,1\right)  \right\}  $ and range
\[
\left\{  z\in\mathbb{R}^{N}:\sum_{i}z_{i}=0\right\}  .
\]

\end{lemma}

\begin{proof}
\textbf{1}. If $N=1$, then $\mathcal{Y}_{1}=P_{1}=\mathbb{R}^{d},$ $\Omega
_{1}=\mathbb{R}\times\mathbb{R}^{d},$ $M_{1}=1$, $B_{1}=\int xd\rho$ and
$G=0.$ The conclusion follows trivially. Henceforth, we assume $N\geqslant2.$
We first consider the case where the effective domain $\left(
\ref{eq:V-effective-domain}\right)  $ is the whole space, i.e. $U=\mathbb{R}
^{d}.$ The extension to $U\neq\mathbb{R}^{d}$ only needs a few additional notes.

\textbf{2}. First we show that $M_{i}$, $B_{i}$ are continuous on $\Omega_{N}
$. Let $(c,y)\in\Omega_{N}$ and $\left(  c^{\left(  n\right)  },y^{\left(
n\right)  }\right)  $ converge to $(c,y)$. Write
\[
P_{i}^{\left(  n\right)  }=P_{i}\left(  c^{\left(  n\right)  },y^{\left(
n\right)  }\right)  ,\text{ }M_{i}^{\left(  n\right)  }=M_{i}\left(
c^{\left(  n\right)  },y^{\left(  n\right)  }\right)  ,\text{ }B_{i}^{\left(
n\right)  }=B_{i}\left(  c^{\left(  n\right)  },y^{\left(  n\right)  }\right)
.
\]
Since $\left\{  y_{i}\right\}  $ are distinct by the definition of $\Omega
_{N}$, the intersections $F_{ij}$ $\left(  i\neq j\right)  $ of distinct cells
has $d-1$ dimension$,$ thus $F=\cup_{i\neq j}F_{ij}$ has null Lebesgue mass.
Outside the null set $F$, the indicators on $P_{i}^{\left(  n\right)  }$
converge pointwise to $P_{i}.$ The dominated convergence theorem therefore
yields
\[
M_{i}^{\left(  n\right)  }\rightarrow M_{i}\text{ and }B_{i}^{\left(
n\right)  }\rightarrow B_{i},
\]
showing the desired continuity. This also shows that $\Omega_{N}$ is open.

\textbf{3}. Fix $i\in\left\{  1,\ldots,N\right\}  $, $(c,y)\in\Omega_{N}.$ To
prove that $M_{i}$, $B_{i}$ are differentiable on $\Omega_{N},$ it is
sufficient to show that their directional derivatives are continuous in
$\Omega_{N}.$ Thus consider non-zero directional vectors $q\in\left(
\mathbb{R}^{d}\right)  ^{N},$ $v\in\mathbb{R}^{d}.$ For sufficiently small
$\epsilon>0,$ consider the perturbations $c^{\epsilon}=c+\epsilon
q,y^{\epsilon}=y+\epsilon v$ and the perturbed cells $P_{i}^{\epsilon}
=P_{i}\left(  c^{\epsilon},y^{\epsilon}\right)  .$ Define
\[
\ell_{ij}\left(  x\right)  =x\cdot\left(  y_{i}-y_{j}\right)  +c_{i}
-c_{j},\text{ }\ell_{ij}^{\epsilon}\left(  x\right)  =\ell_{ij}\left(
x\right)  +\epsilon\left[  x\cdot\left(  v_{i}-v_{j}\right)  +q_{i}
-q_{j}\right]  ,
\]
and
\[
H_{ij}\left(  x\right)  =1_{\{\ell_{ij}\left(  x\right)  \geqslant0\}},\text{
}H_{ij}^{\epsilon}\left(  x\right)  =1_{\{\ell_{ij}^{\epsilon}\left(
x\right)  \geqslant0\}}.
\]
For a scalar function $\psi=1$ or $\psi(x)=x_{k}$ with $1\leqslant k\leqslant
d$, define
\[
I_{i,\psi}\left(  c,y\right)  =\int_{P_{i}}\psi\left(  x\right)  f\left(
x\right)  dx.
\]
Thus $I_{i,1}=M_{i}$ and $I_{i,x_{k}}=\left(  B_{i}\right)  _{k}$. The
derivatives of $M_{i}$, $B_{i}$ along the direction $\left(  q,v\right)  $
are
\begin{equation}
\frac{I_{i,\psi}\left(  c^{\epsilon},y^{\epsilon}\right)  -I_{i,\psi}\left(
c,y\right)  }{\epsilon}=\frac{1}{\epsilon}\int\psi\left(  x\right)  f\left(
x\right)  \left(  1_{P_{i}^{\epsilon}}-1_{P_{i}}\right)  \left(  x\right)  dx.
\label{eq:cell-quotient}
\end{equation}
Using Lemma~\ref{lem:alg-product} we can write
\begin{equation}
1_{P_{i}^{\epsilon}}-1_{P_{i}}=\prod_{k\neq i}H_{ik}^{\epsilon}-\prod_{k\neq
i}H_{ik}=\sum_{j\neq i}\left(  H_{ij}^{\epsilon}-H_{ij}\right)  R_{ij}
^{\epsilon}, \label{eq:indicator-telescope}
\end{equation}
where
\[
R_{ij}^{\epsilon}=\prod_{k\neq i,k<j}H_{ik}^{\epsilon}\prod_{k\neq
i,k>j}H_{ik}.
\]
The convention is that an empty product equals one. Substituting $\left(
\ref{eq:indicator-telescope}\right)  $ into $\left(  \ref{eq:cell-quotient}
\right)  $ yields
\begin{equation}
\frac{I_{i,\psi}\left(  c^{\epsilon},y^{\epsilon}\right)  -I_{i,\psi}\left(
c,y\right)  }{\epsilon}=\sum_{j\neq i}D_{ij,\psi}\left(  \epsilon\right)  ,
\label{eq:quotient-reduction}
\end{equation}
where
\[
D_{ij,\psi}\left(  \epsilon\right)  =\frac{1}{\epsilon}\int\psi\left(
x\right)  f\left(  x\right)  \left(  H_{ij}^{\epsilon}-H_{ij}\right)  \left(
x\right)  R_{ij}^{\epsilon}\left(  x\right)  dx.
\]
Since the sum $\left(  \ref{eq:quotient-reduction}\right)  $ only has finite
terms, it is enough to compute the limit of individual term $D_{ij,\psi
}\left(  \epsilon\right)  $ as $\epsilon\rightarrow0$.

\textbf{4}. To take limit $\epsilon\rightarrow0$, we analyze the integrand to
determine its pointwise limit and dominant function. In the following fix
$j\neq i$, thus $\left\vert y_{i}-y_{j}\right\vert >0.$ Let $n_{ij}$ be the
unit direction of $y_{i}-y_{j},$
\[
n_{ij}=\frac{y_{i}-y_{j}}{\left\vert y_{i}-y_{j}\right\vert }.
\]
So each $x\in\mathbb{R}^{d}$ has a unique representation $x=z+tn_{ij}$ with
$z\cdot n_{ij}=0$ and $t\in\mathbb{R}.$ Write $n_{ij}^{\perp}$ for the space
orthogonal to $n_{ij}$ for short. To figure out the integrand in $D_{ij,\psi
}\left(  \epsilon\right)  $ under the new coordinates, we compute
\begin{align*}
\ell_{ij}^{\epsilon}\left(  x\right)   &  =x\cdot\left(  y_{i}-y_{j}\right)
+\left(  c_{i}-c_{j}\right)  +\epsilon\left[  x\cdot\left(  v_{i}
-v_{j}\right)  +\left(  q_{i}-q_{j}\right)  \right] \\
&  =tn_{ij}\cdot\left(  y_{i}-y_{j}\right)  +\left(  c_{i}-c_{j}\right)
+\epsilon\left[  \left(  z+tn_{ij}\right)  \cdot\left(  v_{i}-v_{j}\right)
+\left(  q_{i}-q_{j}\right)  \right] \\
&  =t\left[  \left\vert y_{i}-y_{j}\right\vert +\epsilon n_{ij}\cdot\left(
v_{i}-v_{j}\right)  \right]  +\epsilon z\cdot\left(  v_{i}-v_{j}\right)
+\left(  c_{i}-c_{j}\right)  +\epsilon\left(  q_{i}-q_{j}\right)
\end{align*}
So $\ell_{ij}^{\epsilon}\left(  x\right)  \geqslant0$ is equivalent to
\[
t\geqslant t_{ij}^{\epsilon}\left(  z\right)  \triangleq-\frac{\left(
c_{i}-c_{j}\right)  +\epsilon\left(  q_{i}-q_{j}\right)  +\epsilon
z\cdot\left(  v_{i}-v_{j}\right)  }{\left\vert y_{i}-y_{j}\right\vert
+\epsilon n_{ij}\cdot\left(  v_{i}-v_{j}\right)  }.
\]
Recall that $v\in\mathbb{R}^{d}$ is the non-zero directional vector introduced
early. To make sure the denominator does not hit zero, we require that
$\epsilon$ be small enough so that
\begin{equation}
\epsilon<\frac{1}{2}\frac{\min_{j\neq i}\left\{  \left\vert y_{i}
-y_{j}\right\vert \right\}  }{\max_{j\neq i}\left\{  \left\vert n_{ij}
\cdot\left(  v_{i}-v_{j}\right)  \right\vert \right\}  }.
\label{eq:max-epsilon}
\end{equation}
We also define
\[
t_{ij}\triangleq t_{ij}^{0}\left(  z\right)  =-\frac{c_{i}-c_{j}}{\left\vert
y_{i}-y_{j}\right\vert }.
\]
So under the new coordinates,
\begin{equation}
H_{ij}=1_{\{t\geqslant t_{ij}\}},\text{ }H_{ij}^{\epsilon}=1_{\{t\geqslant
t_{ij}^{\epsilon}\}}. \label{eq:plane-new-coord}
\end{equation}
Writing $F_{j}^{\epsilon}\left(  x\right)  =\psi\left(  x\right)  f\left(
x\right)  R_{j}^{\epsilon}\left(  x\right)  ,$ we have that $F_{j}^{\epsilon
}\left(  1_{\{t\geqslant t_{ij}^{\epsilon}\left(  z\right)  \}}
-1_{\{t\geqslant t_{ij}\}}\right)  $ is integrable and that
\begin{align*}
D_{ij,\psi}\left(  \epsilon\right)   &  =\frac{1}{\epsilon}\int F_{j}
^{\epsilon}\left(  x\right)  \left(  1_{\{\ell_{ij}^{\epsilon}\left(
x\right)  \geqslant0\}}-1_{\{\ell_{ij}\left(  x\right)  \geqslant0\}}\right)
dx\\
&  =\frac{1}{\epsilon}\int_{n_{ij}^{\perp}}\int_{\mathbb{R}}F_{j}^{\epsilon
}\left(  z+tn_{ij}\right)  \left(  1_{\{t\geqslant t_{ij}^{\epsilon}\left(
z\right)  \}}-1_{\{t\geqslant t_{ij}\}}\right)  dtdz\\
&  =\frac{1}{\epsilon}\int_{n_{ij}^{\perp}}\int_{t_{ij}^{\epsilon}\left(
z\right)  }^{t_{ij}}F_{j}^{\epsilon}\left(  z+tn_{ij}\right)  dtdz\\
&  =\int_{n_{ij}^{\perp}}\frac{t_{ij}-t_{ij}^{\epsilon}\left(  z\right)
}{\epsilon}\int_{0}^{1}F_{j}^{\epsilon}\left(  z+\left[  t_{ij}+s\left(
t_{ij}^{\epsilon}\left(  z\right)  -t_{ij}\right)  \right]  n_{ij}\right)
dsdz.
\end{align*}
If $x$ is in $\left\{  \ell_{ij}\left(  x\right)  \neq0\right\}  ,$ then the
difference $1_{\{\ell_{ij}^{\epsilon}\left(  x\right)  \geqslant0\}}
-1_{\{\ell_{ij}\left(  x\right)  \geqslant0\}}$ will eventually become zero
when $\epsilon\rightarrow0.$ So for possible non-zero pointwise limit, $x$
must be on the plane $\left\{  \ell_{ij}\left(  x\right)  =0\right\}  ,$
equivalently $z+tn_{ij}$ is on the plane $\left\{  \ell_{ij}\left(  x\right)
=0\right\}  .$ So the following analysis is carried out on $\left\{  \ell
_{ij}\left(  x\right)  =0\right\}  $. Write $x_{s}^{\epsilon}=z+\left[
t_{ij}+s\left(  t_{ij}^{\epsilon}\left(  z\right)  -t_{ij}\right)  \right]
n_{ij}.$ We next determine the pointwise limit of $F_{j}^{\epsilon}\left(
x_{s}^{\epsilon}\right)  =\psi\left(  x_{s}^{\epsilon}\right)  f\left(
x_{s}^{\epsilon}\right)  R_{j}^{\epsilon}\left(  x_{s}^{\epsilon}\right)  $ as
$\epsilon\rightarrow0$. For each $z\in n_{ij}^{\perp},$ the bound $\left(
\ref{eq:t-epsilon-bound}\right)  $ indicates that $x_{s}^{\epsilon}$ converges
uniformly to $z+t_{ij}n_{ij}.$ The geometry of hyperplanes will help find the
limit of $R_{j}^{\epsilon}\left(  x_{s}^{\epsilon}\right)  $. Observe that for
any distinct tripple $i,j,k,$
\[
\left\{  \ell_{ij}=0\right\}  \cap\left\{  \ell_{ik}=0\right\}  \text{ has
zero }\mathcal{H}^{d-1}\text{-measure.}
\]
Indeed, the intersection of three or more $\mathbb{R}^{d}$-hyperplanes that
have distinct normal directions has zero $\mathcal{H}^{d-1}$-measure. So, up
to the $\mathcal{H}^{d-1}$-null set, all points on $\left\{  \ell
_{ij}=0\right\}  $ may be regarded as being strictly away from other
hyperplanes$.$ Thus when $\epsilon\rightarrow0,$ the $x_{s}^{\epsilon}$
eventually stays strictly away from other hyperplanes$.$ Therefore if
$z+t_{ij}n_{ij}$ is outside the preceding null set, then
\begin{equation}
R_{j}^{\epsilon}\left(  x_{s}^{\epsilon}\right)  =\prod_{k\neq i,k<j}
1_{\left\{  \ell_{ik}^{\epsilon}\left(  x_{s}^{\epsilon}\right)
\geqslant0\right\}  }\prod_{k\neq i,k>j}1_{\left\{  \ell_{ik}\left(
x_{s}^{\epsilon}\right)  \geqslant0\right\}  }\rightarrow\prod_{k\neq
i,j}1_{\left\{  \ell_{ik}\left(  z+t_{ij}n_{ij}\right)  >0\right\}  }.
\label{eq:R-epsilon-limit}
\end{equation}
Recall that $z+t_{ij}n_{ij}$ is restricted to $\left\{  \ell_{ij}=0\right\}
.$ On this plane, the limit is $1$ if and only if $z+t_{ij}n_{ij}$ is strictly
on the positive side of all hyperplanes $\left\{  \ell_{ik}=0\right\}  $ with
$k\neq i,j$, in other words, if and only if $z+t_{ij}n_{ij}\in F_{ij}.$ Thus,
up to a $\mathcal{H}^{d-1}$-null set, the limit equals $1_{F_{ij}}$.

\textbf{5}. We now determine the dominant function of the integrand in the
integration of $D_{ij,\psi}\left(  \epsilon\right)  $. Straightforward
calculations yield
\begin{align*}
\frac{t_{ij}-t_{ij}^{\epsilon}\left(  z\right)  }{\epsilon}  &  =\frac
{1}{\epsilon}\left[  \frac{\left(  c_{i}-c_{j}\right)  +\epsilon\left(
q_{i}-q_{j}\right)  +\epsilon z\cdot\left(  v_{i}-v_{j}\right)  }{\left\vert
y_{i}-y_{j}\right\vert +\epsilon n_{j}\cdot\left(  v_{i}-v_{j}\right)  }
-\frac{c_{i}-c_{j}}{\left\vert y_{i}-y_{j}\right\vert }\right] \\
&  =\frac{\left\vert y_{i}-y_{j}\right\vert \left(  q_{i}-q_{j}\right)
+\left\vert y_{i}-y_{j}\right\vert z\cdot\left(  v_{i}-v_{j}\right)
-n_{ij}\cdot\left(  v_{i}-v_{j}\right)  \left(  c_{i}-c_{j}\right)  }{\left[
\left\vert y_{i}-y_{j}\right\vert +\epsilon n_{ij}\cdot\left(  v_{i}
-v_{j}\right)  \right]  \left\vert y_{i}-y_{j}\right\vert }\\
&  =\frac{\left(  q_{i}-q_{j}\right)  +\left(  v_{i}-v_{j}\right)  \left[
z+t_{ij}n_{ij}\right]  }{\left\vert y_{i}-y_{j}\right\vert +\epsilon
n_{ij}\cdot\left(  v_{i}-v_{j}\right)  }.
\end{align*}
So by the definition of $t_{ij},t_{ij}^{\epsilon}\left(  z\right)  ,$ there is
a constant $C_{c,y,q,v}>0$ depending on $c,y,q,v$ such that
\begin{equation}
\left\vert t_{ij}\right\vert +\left\vert t_{ij}^{\epsilon}\left(  z\right)
\right\vert \leqslant C_{c,y,q,v}\left(  1+\left\vert z\right\vert \right)
\label{eq:t-epsilon-bound}
\end{equation}
and
\[
\left\vert \frac{t_{ij}-t_{ij}^{\epsilon}\left(  z\right)  }{\epsilon
}\right\vert \leqslant C_{c,y,q,v}\left(  1+\left\vert z\right\vert \right)
.
\]
Let $x_{\ast}$ be from $\left(  \ref{eq:Gaussian-envelope}\right)  $ and
$z_{\ast}$ is the projection of $x_{\ast}$ in the direction orthogonal to
$n_{ij}.$ Then expanding $\left\vert z-2z_{\ast}\right\vert ^{2}/2\geqslant0$
and rearranging yield $\left\vert z-z_{\ast}\right\vert ^{2}\geqslant
\left\vert z\right\vert ^{2}/2-\left\vert z_{\ast}\right\vert ^{2}.$ Hence we
have
\[
\left\vert x_{s}^{\epsilon}-x_{\ast}\right\vert ^{2}\geqslant\left\vert
z-z_{\ast}\right\vert ^{2}\geqslant\left\vert z\right\vert ^{2}/2-\left\vert
x_{\ast}\right\vert ^{2}.
\]
Thus by $\left(  \ref{eq:Gaussian-envelope}\right)  ,$
\[
f\left(  x_{s}^{\epsilon}\right)  \leqslant Z^{-1}e^{-V\left(  x_{\ast
}\right)  }e^{-\kappa\left\vert x-x_{\ast}\right\vert ^{2}/2}\leqslant
Z^{-1}e^{-V\left(  x_{\ast}\right)  -\kappa\left\vert x_{\ast}\right\vert
^{2}/2}e^{-\kappa\left\vert z\right\vert ^{2}/4}.
\]
The function $\psi\left(  x\right)  $ with all of its possible choice has
$\left\vert \psi\left(  x\right)  \right\vert \leqslant1+\left\vert
x\right\vert .$ So $\left\vert \psi\left(  x_{s}^{\epsilon}\right)
\right\vert \leqslant1+\left\vert z\right\vert +\left\vert t_{ij}\right\vert
+\left\vert t_{ij}^{\epsilon}\left(  z\right)  \right\vert $. Also $\left(
\ref{eq:R-epsilon-limit}\right)  $ provides a constant bound on $R_{j}
^{\epsilon}\left(  x_{s}^{\epsilon}\right)  .$ Putting together these bounds,
we obtain, for some constant $C>0$ which depends on $C_{c,y,q,v},V,x_{\ast
},\kappa$, that
\begin{equation}
\left\vert F_{j}^{\epsilon}\left(  x_{s}^{\epsilon}\right)  \right\vert
=\left\vert \psi\left(  x_{s}^{\epsilon}\right)  f\left(  x_{s}^{\epsilon
}\right)  R_{j}^{\epsilon}\left(  x_{s}^{\epsilon}\right)  \right\vert
\leqslant C\left(  1+\left\vert z\right\vert \right)  ^{2}e^{-\kappa\left\vert
z\right\vert ^{2}/4}. \label{eq:strip-majorant}
\end{equation}
With these estimates and the pointwise limit determined earlier, the dominated
convergence theorem therefore gives, as $\epsilon\rightarrow0,$
\[
D_{ij,\psi}\left(  \epsilon\right)  \rightarrow\int_{F_{ij}}\frac{\left(
q_{i}-q_{j}\right)  +\left(  v_{i}-v_{j}\right)  x}{\left\vert y_{i}
-y_{j}\right\vert }\psi(x)f\left(  x\right)  d\mathcal{H}^{d-1}\left(
x\right)  .
\]
Thus $\left(  \ref{eq:mass-derivative}\right)  $ and $\left(
\ref{eq:mm-derivative}\right)  $ are proved.

\textbf{6}. The continuity of the derivatives are determined analogously. Keep
the notations $c^{\left(  n\right)  },y^{\left(  n\right)  },P_{i}^{\left(
n\right)  }$ etc. defined earlier in step 2. Let $F_{ij}^{\left(  n\right)
}=P_{i}^{\left(  n\right)  }\cap P_{j}^{\left(  n\right)  }$. Denote the
perturbed and unperturbed intersection
\[
I_{ij}^{\left(  n\right)  }=\left\{  x\cdot\left(  y_{i}^{\left(  n\right)
}-y_{j}^{\left(  n\right)  }\right)  +c_{i}^{\left(  n\right)  }
-c_{j}^{\left(  n\right)  }=0\right\}  ,\text{ }I_{ij}=\left\{  x\cdot\left(
y_{i}-y_{j}\right)  +c_{i}-c_{j}=0\right\}  .
\]
These are hyperplane, there is a linear transform that takes $I_{ij}$ to
$I_{ij}^{\left(  n\right)  }$ which is denoted by $\xi^{\left(  n\right)  }.$
Its Jacobian determinant is written $J_{\xi^{\left(  n\right)  }}$. Then the
$\left[  q,v\right]  $-directional derivative under perturbation equals
\begin{align*}
&  \int_{F_{ij}^{\left(  n\right)  }}\frac{\left(  q_{i}-q_{j}\right)
+\left(  v_{i}-v_{j}\right)  x}{\left\vert y_{i}^{\left(  n\right)  }
-y_{j}^{\left(  n\right)  }\right\vert }\psi(x)f\left(  x\right)
d\mathcal{H}^{d-1}\left(  x\right) \\
&  =\int_{I_{ij}^{\left(  n\right)  }}1_{F_{ij}^{\left(  n\right)  }}
\frac{\left(  q_{i}-q_{j}\right)  +\left(  v_{i}-v_{j}\right)  x}{\left\vert
y_{i}^{\left(  n\right)  }-y_{j}^{\left(  n\right)  }\right\vert }
\psi(x)f\left(  x\right)  d\mathcal{H}^{d-1}\left(  x\right) \\
&  =\int_{I_{ij}}1_{F_{ij}^{\left(  n\right)  }}\frac{\left(  q_{i}
-q_{j}\right)  +\left(  v_{i}-v_{j}\right)  \xi^{\left(  n\right)  }\left(
x^{\prime}\right)  }{\left\vert y_{i}^{\left(  n\right)  }-y_{j}^{\left(
n\right)  }\right\vert }\psi(\xi^{\left(  n\right)  }\left(  x^{\prime
}\right)  )f\left(  \xi^{\left(  n\right)  }\left(  x^{\prime}\right)
\right)  \left\vert J_{\xi^{\left(  n\right)  }}\left(  x^{\prime}\right)
\right\vert d\mathcal{H}^{d-1}\left(  x^{\prime}\right)  .
\end{align*}
The Jacobian determinant $\left\vert J_{\xi^{\left(  n\right)  }}\left(
x^{\prime}\right)  \right\vert \rightarrow1$ as $n\rightarrow\infty$ (This is
can be computed directly or obtained by noting that the linear transform
$\xi^{\left(  n\right)  }$ covnerges locally uniformly to the identity
transform). The integrand now has a dominant function as earlier. So,
employing the dominated convergent theorem again, we see that the directional
derivatives are continuous in $\Omega_{N}$. Thus $M_{i},$ $B_{i}\in
C^{1}\left(  \Omega_{N}\right)  .$

\textbf{7}. For $q\in\mathbb{R}^{N}$, by the derivative formula just proved we
have
\[
\left(  Gq\right)  _{i}=DM_{i}\left[  q,0\right]  =\sum_{j\neq i}\int_{F_{ij}
}\frac{q_{i}-q_{j}}{\left\vert y_{i}-y_{j}\right\vert }f\left(  x\right)
d\mathcal{H}^{d-1}\left(  x\right)  .
\]
Writing for $i\neq j,$ denote the dominating face integral by
\[
g_{ij}=\int_{F_{ij}}\frac{f\left(  x\right)  }{\left\vert y_{i}-y_{j}
\right\vert }d\mathcal{H}^{d-1}\left(  x\right)  .
\]
Since $G=\left(  g_{ij}\right)  $ is symmetric,
\[
\sum_{i}\left(  Gq\right)  _{i}=\sum_{i}\sum_{j\neq i}g_{ij}\left(
q_{i}-q_{j}\right)  =\sum_{i}\sum_{j\neq i}g_{ij}q_{i}-\sum_{j}\sum_{i\neq
j}g_{ij}q_{j}=0
\]
and
\[
q\cdot Gq=\sum_{i}\sum_{j\neq i}g_{ij}\left(  q_{i}^{2}-q_{i}q_{j}\right)
=\frac{1}{2}\sum_{i}\sum_{j\neq i}g_{ij}\left(  q_{i}-q_{j}\right)  ^{2}
=\sum_{i<j}g_{ij}\left(  q_{i}-q_{j}\right)  ^{2}.
\]
To see that $Gq=0$ implies $q_{i}=q_{j}$ for all $i\neq j,$ we first note if
$F_{ij}$ with $i\neq j$ contains a point of $U$ (see $\left(
\ref{eq:V-effective-domain}\right)  $), then in the neighbourhood of this
point the density $f\left(  x\right)  $ has a positive lower bound, thus we
have $g_{ij}>0.$ If $g_{ij}>0$ wth $i\neq j$ happens, we say that $i$ and $j$
are \textit{connected}; we say that $i$ and $j$ are \textit{connected by a
path} if there is finite sequence of indices so that every consecutive pair of
indices are connected in the sense just defined. Write $\ell_{i}\left(
x\right)  =x\cdot y_{i}+c_{i},$ $i=1,...,N.$ The site vectors $y_{i}$ are
distinct, so every distinct pair of planes $\ell_{i},$ $\ell_{j}$ intersect.
But for the intersection to locate on $F_{ij}$, the pair of planes $\ell_{i},$
$\ell_{j}$ must at the same time dominate all the remaining planes. We claim
that for every pair $i\neq j$ are connected by a path. Let $S\subset\left\{
1,...,N\right\}  $ be any non-empty proper subset. Since $M_{k}>0$ for all
$k\in\left\{  1,...,N\right\}  ,$ we infter that, for every pair $i\in S,$
$j\in S^{c},$ there are distinct points $a,$ $b\in U$ such that $\ell_{i}$
strictly dominates all other planes at $a$ and $\ell_{j}$ strictly dominates
all other planes at $b.$ Let $n_{a,b}=\left\vert b-a\right\vert ^{-1}\left(
b-a\right)  .$ By continuity, there exists a small disk perpendicular to
$n_{a,b},$
\[
D_{r}=\left\{  z\in\mathbb{R}^{d}:\left\vert z\right\vert \leqslant r,\text{
}z\cdot n_{a,b}=0\right\}  ,\text{ }r>0,
\]
such that $a+D_{r},$ $b+D_{r}\subset U,$ $\ell_{i}$ still strictly dominates
all other planes on $a+D_{r}$ and $\ell_{j}$ still strictly dominates all
other planes on $b+D_{r}.$ Therefore, for each $z\in D_{r},$ the continuous
function
\[
L\left(  t\right)  =\max_{i^{\prime}\in S}\ell_{i^{\prime}}\left(
a+z+t\left(  b-a\right)  \right)  -\max_{j^{\prime}\in S^{c}}\ell_{j^{\prime}
}\left(  a+z+t\left(  b-a\right)  \right)  ,\text{ }\forall t\in\left[
0,1\right]  ,
\]
satisfies $L\left(  0\right)  >0>L\left(  1\right)  .$ Thus $L\left(  t_{\ast
}\right)  =0$ for some $i^{\prime}\in S,$ $j^{\prime}\in S^{c}$ and some
$t_{\ast}\in\left(  0,1\right)  $. So $a+z+t_{\ast}\left(  b-a\right)  \in U$
is on $F_{i^{\prime}j^{\prime}}.$ If we move the disk $D_{r}$ from $a$ to $b,$
it will sweep across a cylinder $K_{a,b}$ with axis $b-a.$ Then the projection
of $\bigcup_{i\in S,j\in S^{c}}\left(  F_{ij}\cap K_{a,b}\right)  $ onto the
disk $a+D_{r}$ equals the disk itself. Moreover the cylinder is contained in
$U,$ so the density $f\left(  x\right)  $ has a positive lower bound on it.
Therefore
\[
\sum_{i\in S,j\in S^{c}}\left\vert y_{i}-y_{j}\right\vert g_{ij}\geqslant
\min_{x\in K_{a,b}}f\left(  x\right)  \sum_{i\in S,j\in S^{c}}\mathcal{H}
^{d-1}\left(  F_{ij}\cap K_{a,b}\right)  \geqslant\min_{x\in K_{a,b}}f\left(
x\right)  \mathcal{H}^{d-1}\left(  a+D_{r}\right)  >0.
\]
The second inequality is due to the fact that volumn decreases after
projection. So every pair $i\in S,$ $j\in S^{c}$ is connected by a path. We
now can start with a single index set $S=\left\{  i\right\}  $ and expand it
by adding index outside $S$ that connects back to $S$ one at a time, until all
indices are absorbed. Therefore we have proved that every distinct pair $i,$
$j$ is connected by a path. With this connectedness, the equation
\[
Gq=\sum_{i<j}g_{ij}\left(  q_{i}-q_{j}\right)  ^{2}=0
\]
will lead to the conclusion that $q_{i}=q_{j}$ for all $i\neq j$. Thus
$\ker\left(  G\right)  =\operatorname{span}\left\{  \left(  1,\ldots,1\right)
\right\}  .$ Again, using that $G$ is symmetric, we obtain that the range of
$G$ equals $\left(  \ker G\right)  ^{\perp}=\left\{  z:\sum_{i}z_{i}
=0\right\}  .$

\textbf{8}. We have completed the proof for $U=\mathbb{R}^{d}$. The extension
to general needs no special treatment. But note that $U$ is convex and has
null boundary in $\mathbb{R}^{d}.$ When surface integral is to compute, there
will be intersections with the boundary of $U$, for example $\left(  \partial
U\right)  \cap I_{ij}$ will replace $I_{ij}$ in the computation of directional
derivatives. Since $U$ and $I_{ij}$ are convex, $\left(  \partial U\right)
\cap I_{ij}$ is in the relative boundary of the convex set $U\cap I_{ij},$ so
\[
\mathcal{H}^{d-1}\left(  \left(  \partial U\right)  \cap I_{ij}\right)  =0.
\]
Therefore all previous computations continue to hold when $U$ is not equal to
$\mathbb{R}^{d}$. The proof is completed.
\end{proof}

The next lemma constructs cells with prescribed masses by transportation.

\begin{lemma}
[Prescribed positive masses]\label{lem:weighted-masses}Assume that $\left(
\ref{eq:strong-source}\right)  $ holds. Fix $N\in\mathbb{N}$, $\left(
y_{1},\ldots,y_{N}\right)  \in\mathcal{Y}_{N}$ and $\left(  m_{i}\right)
\in\mathbb{R}^{N},$ $m_{i}>0$ with $\sum_{i}m_{i}=1$. Then there is a unique
vector $c\in\mathbb{R}^{N}$ such that
\begin{equation}
\sum_{i}m_{i}c_{i}=0\text{ and }M_{i}(c,y)=m_{i},\text{ }i\in\left\{
1,...,N\right\}  . \label{eq:prescribed}
\end{equation}
Moreover, for these fixed masss, let
\[
H_{m}=\left\{  c\in\mathbb{R}^{N}:\sum_{i}m_{i}c_{i}=0\right\}  .
\]
Then the solution map $\mathcal{S}_{N}:\mathcal{Y}_{N}\mapsto H_{m}$ is
$C^{1}$ on the open set $\mathcal{Y}_{N}$.
\end{lemma}

\begin{proof}
\textbf{1}. Consider the function on $\mathbb{R}^{N},$
\[
F\left(  c\right)  =\int\max_{i}\left\{  x\cdot y_{i}+c_{i}\right\}
d\rho\left(  x\right)  -\sum_{i}m_{i}c_{i}.
\]
It is finite convex, thus is continuous on $\mathbb{R}^{N}.$ We will show that
$F$ admits a global minimum on the constrained space $H_{m}$ and the
stationary equation at the global minimum gives exactly the desired solution.
The first step toward that goal is computating its derivative$.$ The $\max$ in
the integrand is taken over a finite number of linear functions. Let
$D_{i}=\operatorname*{int}P_{i}\left(  c,y\right)  .$ Define
\[
E=\mathbb{R}^{d}\backslash\bigcup_{i}D_{i}.
\]
Here $P_{i}\left(  c,y\right)  $ is defined as in $\left(  \ref{eq:Def-Pi}
\right)  .$ Clearly $E$ has zero Lebesgue measure. Write $\ell_{i}\left(
x\right)  =x\cdot y_{i}+c_{i},$ $i=1,...,N.$ For every $x$ outside $E,$
$\max_{i}\left\{  x\cdot y_{i}+c_{i}\right\}  $ is attained at a unique
maximizing index $i\left(  x\right)  $,
\[
\max_{i}\left\{  x\cdot y_{i}+c_{i}\right\}  =x\cdot y_{i\left(  x\right)
}+c_{i\left(  x\right)  }>\max_{j\neq i\left(  x\right)  }\left\{  x\cdot
y_{j}+c_{j}\right\}  .
\]
Equivalently
\[
\min_{j\neq i\left(  x\right)  }\left\{  \ell_{i\left(  x\right)  }\left(
x\right)  -\ell_{j}\left(  x\right)  \right\}  >0.
\]
Since $E^{c}$ is open, a small perturnation of the linear functions will not
change the value of $i\left(  x\right)  .$ Precisely, consider any directional
vector $q\in\mathbb{R}^{N}$. For any $j\neq i\left(  x\right)  ,$
\begin{align*}
x\cdot y_{i\left(  x\right)  }+c_{i\left(  x\right)  }+q_{i\left(  x\right)
}-\left(  x\cdot y_{j}+c_{j}+q_{j}\right)   &  =\left[  \ell_{i\left(
x\right)  }\left(  x\right)  -\ell_{j}\left(  x\right)  \right]  -2\left\Vert
q\right\Vert _{\infty}\\
&  \geqslant\min_{j\neq i\left(  x\right)  }\left\{  \ell_{i\left(  x\right)
}\left(  x\right)  -\ell_{j}\left(  x\right)  \right\}  -2\left\Vert
q\right\Vert _{\infty}>0.
\end{align*}
The last step of the inequality is valid provided $\left\Vert q\right\Vert
_{\infty}$ is small. Hence for any sufficiently small $t>0,$
\[
\frac{\max_{i}\left\{  x\cdot y_{i}+c_{i}+tq_{i}\right\}  -\max_{i}\left\{
x\cdot y_{i}+c_{i}\right\}  }{t}=\frac{x\cdot y_{i\left(  x\right)
}+c_{i\left(  x\right)  }+tq_{i\left(  x\right)  }-\left[  x\cdot y_{i\left(
x\right)  }+c_{i\left(  x\right)  }\right]  }{t}=q_{i\left(  x\right)  }.
\]
The difference quotient also has an integrable control, because
\[
\left\vert \max_{i}\left\{  x\cdot y_{i}+c_{i}+tq_{i}\right\}  -\max
_{i}\left\{  x\cdot y_{i}+c_{i}\right\}  \right\vert \leqslant t\left\Vert
q\right\Vert _{\infty}.
\]
Therefore by the dominated convergence theorem we have
\begin{align*}
DF\left(  c\right)  \left[  q\right]   &  =\lim_{t\rightarrow0}\frac{F\left(
c+tq\right)  -F\left(  c\right)  }{t}=\int q_{i\left(  x\right)  }d\rho
-\sum_{i}m_{i}q_{i}\\
&  =\sum_{i}\int_{P_{i}}q_{i}d\rho-\sum_{i}m_{i}q_{i}=\sum_{i}\left[
M_{i}\left(  c,y\right)  -m_{i}\right]  q_{i}.
\end{align*}
The integrations may be regarded as being carried out on $\mathbb{R}
^{N}\backslash E$ because $E$ has zero $\rho$-mass. Thus the splitting of the
integral into sums of sub-domain integrals are justified. For each $x\in
P_{i}\backslash E,$ the maximizing index $i\left(  x\right)  $ exactly equal
$i$. So this computation is valid and yields the directional derivative of
$F\left(  \cdot\right)  $. More can be said. $F\left(  \cdot\right)  $ is
actually differentiable. Indeed, for any $q\in\mathbb{R}^{N},$
\[
F\left(  c+q\right)  -F\left(  c\right)  -\sum_{i}\left[  M_{i}\left(
c,y\right)  -m_{i}\right]  q_{i}=\int\left[  \max_{j}\left\{  x\cdot
y_{j}+c_{j}+q_{j}\right\}  -\max_{j}\left\{  x\cdot y_{j}+c_{j}\right\}
-q_{i\left(  x\right)  }\right]  d\rho.
\]
The integrad
\[
r\left(  x\right)  =\max_{j}\left\{  x\cdot y_{j}+c_{j}+q_{j}\right\}
-\max_{j}\left\{  x\cdot y_{j}+c_{j}\right\}  -q_{i\left(  x\right)  }
\]
has $\left\vert r\left(  x\right)  \right\vert \leqslant2\left\Vert
q\right\Vert _{\infty}$. Also observe that the previous analysis implies that
for each $x\in E^{c},$ $r\left(  x\right)  =0$ as soon as$\left\Vert
q\right\Vert _{\infty}$ is small enough. Therefore the pointwise limit holds
\[
\lim_{\left\Vert q\right\Vert _{\infty}\rightarrow0}\frac{\left\vert r\left(
x\right)  \right\vert }{\left\Vert q\right\Vert _{\infty}}=0,\text{ }\forall
x\in E^{c}.
\]
It follows that
\[
\left\vert \frac{F\left(  c+q\right)  -F\left(  c\right)  -\sum_{i}\left[
M_{i}\left(  c,y\right)  -m_{i}\right]  q_{i}}{\left\Vert q\right\Vert
_{\infty}}\right\vert \leqslant\int\frac{\left\vert r\left(  x\right)
\right\vert }{\left\Vert q\right\Vert _{\infty}}d\rho\rightarrow0\text{ as
}\left\Vert q\right\Vert _{\infty}\rightarrow0,
\]
showing that $F\left(  \cdot\right)  $ is differentiable at $c\in
\mathbb{R}^{N}.$

\textbf{2}. Now consider the restriction of $F\left(  \cdot\right)  $ to
$H_{m}.$ For any $c\in H_{m},$ let $c_{\max}=\max_{i}c_{i}$. The constraint
implies $c_{\max}\geqslant0$, so
\[
F\left(  c\right)  \geqslant c_{\max}-\max_{i}\left\vert y_{i}\right\vert
\int\left\vert x\right\vert d\rho.
\]
Since
\[
c_{\max}\geqslant c_{i}=-\frac{1}{m_{i}}\sum_{j\neq i}m_{j}c_{j}
\geqslant-\frac{1-m_{i}}{m_{i}}c_{\max}\geqslant-\frac{1}{m_{i}}c_{\max}
\]
and $\min_{i}m_{i}>0,$ we have $\left\Vert c\right\Vert _{\infty}\leqslant
c_{\max}/\left(  \min_{i}m_{i}\right)  <\infty.$ It follows that
\[
F\left(  c\right)  \geqslant\left(  \min_{i}m_{i}\right)  \left\Vert
c\right\Vert _{\infty}-\max_{i}\left\vert y_{i}\right\vert \int\left\vert
x\right\vert d\rho.
\]
This indicates that $F\left(  \cdot\right)  $ is coercive on the linear space
$H_{m}$ equipped with this equivalent nom $\left\Vert \cdot\right\Vert
_{\infty}$ of $\mathbb{R}^{N},$ thus attains its minimum there.

\textbf{3}. Let $c_{\ast}\in H_{m}$ be a \ minimum of $F\left(  \cdot\right)
$ on $H_{m},$ we claim that it is also a global minimum of $F\left(
\cdot\right)  $ on $\mathbb{R}^{N}.$ Indeed, at $c_{\ast}$ we have $DF\left(
c_{\ast}\right)  \left[  q\right]  =0$ for all perturbation vector $q$
satisfying $c_{\ast}+q\in H_{m}.$ Since $H_{m}$ is a linear space, the
requirement means $q$ can be any vector in $H_{m}.$ Thus
\begin{equation}
\sum_{i}\left[  M_{i}\left(  c_{\ast},y\right)  -m_{i}\right]  q_{i}=0\text{
for all }q\in H_{m}. \label{eq:Hm-stationary-point}
\end{equation}
Hence, writing $M=\left(  M_{i}\left(  c,y\right)  \right)  ,$ $m=\left(
m_{i}\right)  ,$ we have, for some $\lambda\in\mathbb{R}$,
\[
M-m=\lambda m.
\]
Then
\[
0=\sum_{i}(M_{i}-m_{i})=\lambda\sum_{i}m_{i}=\lambda.
\]
Hence $M_{i}=m_{i}$ for all $i.$ Consequently, $\left(
\ref{eq:Hm-stationary-point}\right)  $ holds for all $q\in\mathbb{R}^{N},$ not
just those in $H_{m}.$ This indicates that $c_{\ast}$ is global stationary
point of $F\left(  \cdot\right)  .$ Together with the convexity, this shows
that $c_{\ast}$ is a global minimum of $F\left(  \cdot\right)  $ on
$\mathbb{R}^{N}.$

\textbf{4}. We have prove that for any $y_{0}\in\mathcal{Y}_{N},$ there is a
$c_{0}\in\mathbb{R}^{N}$ such that is a global minimum of the associated
functional $F\left(  \cdot\right)  ,$ and $c_{0}$ satisfies the constraint
$\left(  \ref{eq:prescribed}\right)  $. So the mapping $\left(  c,y\right)
\mapsto M\left(  c,y\right)  -m$ on $\mathbb{R}^{N}\times\mathcal{Y}_{N}$
vanishes at $\left(  c_{0},y_{0}\right)  .$ Since $m=\left(  m_{i}\right)
>0$, thus $\left(  c_{0},y_{0}\right)  \in\left(  H_{m}\times\mathcal{Y}
_{N}\right)  \subset\Omega_{N}.$ Lemma \ref{lem:weighted-calculus} shows that
this mapping is $C^{1}$ in an open neighbourhood of $\left(  c_{0}
,y_{0}\right)  .$ Denote by $\Phi$ the restriction of the preceding mapping to
$H_{m}\times\mathcal{Y}_{N},$
\[
\Phi\left(  c,y\right)  =M\left(  c,y\right)  -m.
\]
It is also $C^{1}$ and $D_{c}\Phi\left(  c_{0},y_{0}\right)  =G|_{H_{m}}.$ If
$z\in H_{m}$ and $Gz=0,$ then by the third part of Lemma
\ref{lem:weighted-calculus}, we can write $z=t\left(  1,...,1\right)  $ for
some $t\in\mathbb{R}$ and
\[
0=\sum_{i}m_{i}z_{i}=t\sum_{i}m_{i}=t.
\]
Hence $z=0.$ So $G|_{H_{m}},$ which has kernel $\left\{  0\right\}  $, is a
one-one map from $H_{m}$ to $\left\{  z\in\mathbb{R}^{N}:\sum_{i}
z_{i}=0\right\}  $. Thus $G|_{H_{m}}$ is a $\left(  N-1\right)  $ by $\left(
N-1\right)  $ invertible matrix. The implicit function theorem therefore gives
a neighbourhood about $y_{0}$ in which $c$ is uniquely determined as a $C^{1}$
function of $y.$

\textbf{5}. If $c_{0}$ and $c_{0}^{\prime}$ are two distinct solutions in
$H_{m}$ corresponding to the same $y_{0}$, then the associated convex
functional $F$ constructed in step 1 satisfies
\[
F((1-t)c_{0}+tc_{0}^{\prime})=\min_{H_{m}}F\text{ for all }t\in\left[
0,1\right]  .
\]
As $t\rightarrow0,$ this yields two distinct solutions in the neighbourhood of
$\left(  c_{0},y_{0}\right)  ,$ contradicting local uniqueness at $c_{0}$.
Thus the solution is unique.
\end{proof}

Lemma \ref{lem:weighted-masses} is trivial when $N=1$. Since in this case
$H_{m}=\{0\}$, $P_{1}=\mathbb{R}^{d}$, $M_{1}=1$ and the solution
corresponding to any $y$ is simply $c_{1}=0$.

\section{Estimates of total cell variance}

\label{sec:joint}Recall the abbreviations in $\left(  \ref{eq:short-Pi-Mi-Bi}
\right)  $. Fix $(c,y)\in\Omega_{N}$ and a vector $q\in\mathbb{R}^{N}$, we
write
\begin{equation}
\dot{M}_{i}=DM_{i}\left[  q,0\right]  ,\text{ }\dot{b}_{i}=Db_{i}\left[
q,0\right]  . \label{eq:test}
\end{equation}

\begin{lemma}
\label{lem:joint}Assume that $\left(  \ref{eq:logconcave-source}\right)  $
holds. Fix $\left(  c,y\right)  \in\Omega_{N}$. Let $I_{d}$ be the
$d$-dimensional identity matrix and $M_{i},\Sigma_{i}$ are defined as in
$\left(  \ref{eq:short-Pi-Mi-Bi}\right)  $. Define, any $q\in\mathbb{R}^{N}$
and $\xi\in\left(  \mathbb{R}^{d}\right)  ^{N},$
\begin{equation}
H\left(  q,\xi\right)  =\sum_{i}M_{i}\dot{b}_{i}\cdot\xi_{i},\text{ }D\left(
\xi\right)  =\sum_{i}M_{i}\xi_{i}^{\top}\left(  \kappa^{-1}I_{d}-\Sigma
_{i}\right)  \cdot\xi_{i}, \label{eq:forms}
\end{equation}
and
\[
I(q)=\sum_{i}M_{i}\left(  \dot{M}_{i}/M_{i}\right)  ^{2}.
\]
Then the following hold.

(i) For each $i,$ $\kappa^{-1}I_{d}-\Sigma_{i}$ is positive semidefinite.

(ii) For any $q\in\mathbb{R}^{N}$ and $\xi\in\left(  \mathbb{R}^{d}\right)
^{N}$,
\begin{equation}
2H\left(  q,\xi\right)  \leqslant I\left(  q\right)  +D\left(  \xi\right)
\text{.} \label{eq:joint}
\end{equation}

\end{lemma}

\begin{proof}
\textbf{1}. In the following, we write $P_{i}\left(  c\right)  =P_{i}\left(
c,y\right)  $ and define
\begin{equation}
F_{i}\left(  c,\theta\right)  =e^{-\left\vert \theta\right\vert ^{2}/\left(
2\kappa\right)  }\int_{P_{i}\left(  c\right)  }e^{\theta\cdot x}d\rho\left(
x\right)  . \label{eq:tilt}
\end{equation}
Substituting the density $\left(  \ref{eq:logconcave-source}\right)  $ into
the integral and completing the square, we can write the integrand, up to the
null boundary of $U$ and normalization factor $Z^{-1}$, as
\[
\exp\left[  -\left(  V\left(  x\right)  -\tfrac{\kappa}{2}\left\vert
x\right\vert ^{2}\right)  -\tfrac{\kappa}{2}\left\vert x-\theta/\kappa
\right\vert ^{2}\right]  1_{\left\{  x\cdot\left(  y_{i}-y_{j}\right)
+\left(  c_{i}-c_{j}\right)  \geqslant0,\ \forall j\right\}  }1_{U}.
\]
The exponent is concave in $\left(  x,\theta\right)  $. The linear hyperplane
constraints are linear in $\left(  x,c\right)  $. Therefore the integrand is
log-concavity in $\left(  x,c,\theta\right)  $. Note that $V\left(  x\right)
-\tfrac{\kappa}{2}\left\vert x\right\vert ^{2}$ is convex hence supported from
below by some hyperplane, thus the integrand is integrable w.r.t. $dx$. By
Pr\'{e}kopa's marginal theorem \cite[Theorem 6]{prekopa1973logarithmic} we see
that $F_{i}$ is log-cocave in $\left(  c,\theta\right)  .$ Before proceeding
to the next step, we introduce a few notations. Set
\[
P_{i}\left(  t\right)  =P_{i}\left(  c+tq,y\right)  ,\text{ }M_{i}\left(
t\right)  =M_{i}\left(  c+tq,y\right)  ,\text{ }b_{i}\left(  t\right)
=b_{i}\left(  c+tq,y\right)  ,
\]
and
\[
T_{i}\left(  t,\theta\right)  =\int_{P_{i}\left(  c+tq\right)  }e^{\theta\cdot
x}d\rho\left(  x\right)  ,\text{ }K_{i,t}\left(  \theta\right)  =\log\left[
T_{i}\left(  t,\theta\right)  /M_{i}\left(  t\right)  \right]  .
\]
Since $M_{i}\left(  0\right)  =M_{i}>0,$ so these definitions make sense at
least when $t$ is small and $T_{i}\left(  t,\theta\right)  >0$ if $\theta$ is
also small. To simplify notations, we will also write $T_{i}=T_{i}\left(
t,\theta\right)  $ and $K_{i,t}=K_{i,t}\left(  \theta\right)  .$ Under the
source condition $\left(  \ref{eq:logconcave-source}\right)  ,$ $T_{i}\left(
t,\theta\right)  $ is at least twice differentiable w.r.t. $\theta.$ For
$1\leqslant j,$ $k\leqslant d$,
\[
\partial_{\theta_{j}}T_{i}=\int_{P_{i}\left(  c+tq\right)  }x_{j}
e^{\theta\cdot x}d\rho\left(  x\right)  ,\text{ }\partial_{\theta_{k}}
\partial_{\theta_{j}}T_{i}=\int_{P_{i}\left(  c+tq\right)  }x_{j}
x_{k}e^{\theta\cdot x}d\rho\left(  x\right)  .
\]
Hence
\[
\partial_{\theta_{j}}K_{i,t}=\frac{\partial_{\theta_{j}}T_{i}}{T_{i}},\text{
}\partial_{\theta_{k}}\partial_{\theta_{j}}K_{i,t}=\frac{\partial_{\theta_{k}
}\partial_{\theta_{j}}T_{i}}{T_{i}}-\frac{\partial_{\theta_{k}}T_{i}}{T_{i}
}\frac{\partial_{\theta_{j}}T_{i}}{T_{i}}.
\]
In matrix form
\[
\nabla_{\theta}K_{i,t}=\frac{\nabla_{\theta}T_{i}}{T_{i}},\text{ }
\nabla_{\theta}^{2}K_{i,t}=\frac{\nabla_{\theta}^{2}T_{i}}{T_{i}}
-\frac{\left(  \nabla_{\theta}T_{i}\right)  \left(  \nabla_{\theta}
T_{i}\right)  ^{\top}}{T_{i}^{2}}.
\]
\textbf{2}. By the definition of $F_{i},$ we have
\[
\log F_{i}\left(  c,\theta\right)  -\log F_{i}\left(  c,0\right)  =\log\left[
T_{i}\left(  0,\theta\right)  /M_{i}\right]  -\frac{\kappa^{-1}}{2}\left\vert
\theta\right\vert ^{2}=K_{i,0}\left(  \theta\right)  -\frac{\kappa^{-1}}
{2}\left\vert \theta\right\vert ^{2}.
\]
Its Hessian at $\theta=0$ is
\[
\nabla_{\theta}^{2}K_{i,0}\left(  0\right)  -\kappa^{-1}I_{d}=\left[  \left(
\Sigma_{i}+b_{i}b_{i}^{\top}\right)  -b_{i}b_{i}^{\top}\right]  -\kappa
^{-1}I_{d}=\Sigma_{i}-\kappa^{-1}I_{d}.
\]
Since $F_{i}\left(  c,\theta\right)  $ is log-concave in $\theta$, we get that
$\kappa^{-1}I_{d}-\Sigma_{i}$ is positive semidefinite. This proves the first conclusion.

\textbf{3}. The desired inequalities in $\left(  \ref{eq:joint}\right)  $, for
each given $q\in\mathbb{R}^{N}$ and $\xi\in\left(  \mathbb{R}^{d}\right)
^{N}$, are packed in the function
\[
\Phi\left(  t\right)  =\sum_{i}M_{i}\log F_{i}\left(  c+tq,t\xi_{i}\right)  .
\]
Specifically we will show that the first inequality of $\left(  \ref{eq:joint}
\right)  $ is equivalent to $\Phi^{\prime\prime}\left(  0\right)  \leqslant0.$
Since $F_{i}\left(  c,\theta\right)  $ is log-concave in $\left(
c,\theta\right)  $, $\Phi\left(  t\right)  $ is a concave function of $t.$ So
the equivalence has already been established and the proof boils down to the
computation of $\Phi^{\prime\prime}\left(  0\right)  .$ Noting $F_{i}\left(
c,0\right)  =M_{i},$ we have
\begin{align*}
\Phi\left(  t\right)  -\Phi\left(  0\right)   &  =\sum_{i}M_{i}\log\left[
F_{i}\left(  c+tq,t\xi_{i}\right)  /M_{i}\right]  =\sum_{i}M_{i}\left[
\log\left(  M_{i}\left(  t\right)  /M_{i}\right)  +\log\left(  T_{i}\left(
t,t\xi_{i}\right)  /M_{i}\left(  t\right)  \right)  -\frac{t^{2}\left\vert
\xi_{i}\right\vert ^{2}}{2\kappa}\right] \\
&  =\sum_{i}M_{i}\log\left[  M_{i}\left(  t\right)  /M_{i}\right]  +\sum
_{i}M_{i}K_{i,t}\left(  t\xi_{i}\right)  -\frac{t^{2}}{2\kappa}\sum_{i}
M_{i}\left\vert \xi_{i}\right\vert ^{2}.
\end{align*}
The first term is nothing but the negative Kullback-Leibler divergence between
$M=\left(  M_{i}\right)  $, $M\left(  t\right)  =\left(  M_{i}\left(
t\right)  \right)  .$

\textbf{4}. To analyze the Kullback-Leibler divergence term, we write
\[
\sum_{i}M_{i}\log\left[  M_{i}\left(  t\right)  /M_{i}\right]  =\sum_{i}
M_{i}\log\left[  1+\left(  M_{i}\left(  t\right)  -M_{i}\right)
/M_{i}\right]  .
\]
By Lemma \ref{lem:weighted-calculus}, $M_{i}\left(  t\right)  $ is $C^{1}$ in
$t,$ thus we can write $M_{i}\left(  t\right)  -M_{i}=t\dot{M}_{i}+o\left(
t\right)  .$ Using the Taylor expansion $\log\left(  1+u\right)
=u-u^{2}/2+o\left(  u^{2}\right)  $, we get
\begin{align*}
\sum_{i}M_{i}\log\left[  M_{i}\left(  t\right)  /M_{i}\right]   &  =\sum
_{i}M_{i}\left[  \left(  M_{i}\left(  t\right)  -M_{i}\right)  /M_{i}-\frac
{1}{2}\left(  \left(  M_{i}\left(  t\right)  -M_{i}\right)  /M_{i}\right)
^{2}+o\left(  t^{2}\right)  \right] \\
&  =-\frac{t^{2}}{2}\sum_{i}\left(  \dot{M}_{i}\right)  ^{2}/M_{i}+o\left(
t^{2}\right)  .
\end{align*}

\textbf{5}. To analyze the term with $M_{i}K_{i,t}\left(  t\xi_{i}\right)  ,$
let $\varphi\left(  s\right)  =K_{i,t}\left(  s\theta\right)  .$ Noting
$T_{i}\left(  t,0\right)  =M_{i}\left(  t\right)  ,$ thus $K_{i,t}\left(
0\right)  =0.$ So
\[
\varphi\left(  0\right)  =0,\text{ }\varphi^{\prime}\left(  0\right)
=\nabla_{\theta}K_{i,t}\left(  0\right)  \cdot\theta,\text{ }\varphi
^{\prime\prime}\left(  s\right)  =\theta^{\top}\nabla_{\theta}^{2}
K_{i,t}\left(  s\theta\right)  \cdot\theta.
\]
By the Taylor expansion
\[
\varphi\left(  1\right)  =\varphi\left(  0\right)  +\varphi^{\prime}\left(
0\right)  +\int_{0}^{1}\left(  1-s\right)  \varphi^{\prime\prime}\left(
s\right)  ds.
\]
Hence, substituting the definition of $\varphi$ and $\theta=t\xi_{i}$ into the
equation yields
\[
K_{i,t}\left(  t\xi_{i}\right)  =t\nabla_{\theta}K_{i,t}\left(  0\right)
\cdot\xi_{i}+t^{2}\int_{0}^{1}\left(  1-s\right)  \left[  \xi_{i}^{\top}
\nabla_{\theta}^{2}K_{i,t}\left(  st\xi_{i}\right)  \cdot\xi_{i}\right]  ds.
\]
By the early computations and the differentiability guaranteed by Lemma
\ref{lem:weighted-calculus}, we can write
\[
\nabla_{\theta}K_{i,t}\left(  0\right)  =b_{i}\left(  t\right)  =b_{i}
+t\dot{b}_{i}+o\left(  t^{2}\right)  .
\]
Since $\nabla_{\theta}^{2}K_{i,t}\left(  0\right)  =\Sigma_{i},$ the integral
in $K_{i,t}\left(  t\xi_{i}\right)  $ can be replaced with
\[
t^{2}\int_{0}^{1}\left(  1-s\right)  \left(  \xi_{i}^{\top}\Sigma_{i}\cdot
\xi_{i}\right)  ds+o\left(  t^{2}\right)  =\frac{1}{2}t^{2}\xi_{i}^{\top
}\Sigma_{i}\cdot\xi_{i}+o\left(  t^{2}\right)  .
\]
Substituting these into the equation of $K_{i,t}\left(  t\xi_{i}\right)  $, we
obtain that
\[
K_{i,t}\left(  t\xi_{i}\right)  =tb_{i}\cdot\xi_{i}+t^{2}\dot{b}_{i}\cdot
\xi_{i}+\frac{1}{2}t^{2}\xi_{i}^{\top}\Sigma_{i}\cdot\xi_{i}+o\left(
t^{2}\right)  .
\]

\textbf{6}. Combing these computations and subsituting into the early
equation, we see that $\Phi\left(  t\right)  -\Phi\left(  0\right)  $ equals
\[
-\frac{t^{2}}{2}\sum_{i}\left(  \dot{M}_{i}\right)  ^{2}/M_{i}+\sum_{i}
M_{i}\left[  tb_{i}\cdot\xi_{i}+t^{2}\dot{b}_{i}\cdot\xi_{i}+\frac{1}{2}
t^{2}\xi_{i}^{\top}\Sigma_{i}\cdot\xi_{i}\right]  -\frac{t^{2}}{2\kappa}
\sum_{i}M_{i}\left\vert \xi_{i}\right\vert ^{2}+o\left(  t^{2}\right)  .
\]
The concavity of $\Phi$ implies that the coefficient of second order term must
be negative,
\[
-\frac{1}{2}\sum_{i}\left(  \dot{M}_{i}\right)  ^{2}/M_{i}+\sum_{i}M_{i}
\dot{b}_{i}\cdot\xi_{i}-\frac{1}{2}\sum_{i}M_{i}\left[  \xi_{i}^{\top}\left(
\kappa^{-1}I_{d}-\Sigma_{i}\right)  \cdot\xi_{i}\right]  \leqslant0.
\]
So rearranging yields the first inequality of $\left(  \ref{eq:joint}\right)
,$
\[
\sum_{i}M_{i}\dot{b}_{i}\cdot\xi_{i}\leqslant\sum_{i}M_{i}\left(  \dot{M}
_{i}/M_{i}\right)  ^{2}+\sum_{i}M_{i}\left[  \xi_{i}^{\top}\left(  \kappa
^{-1}I_{d}-\Sigma_{i}\right)  \cdot\xi_{i}\right]  .
\]

\end{proof}

\section{Linear potential stability and conversion to maps}

\label{sec:potential-proof}

For finite targets, an optimal coupling between the targets specifies how much
mass moves between every pair of sites. For optimal transport, the masses move
along no-crossing straight lines. By differentiating the constructed
potential, we reduce the stability to an estimate for its time velocity. The
subsequent normalization of the cell barycenter then turns this time velocity
into a conditional fluctuation and a cell mean. This is essential for an
unbounded source becasue the location of an individual cell need not be
uniformly bounded.

\begin{proposition}
[Finite targets]\label{prop:finite-strong}Assume that $\left(
\ref{eq:strong-source}\right)  $ holds, and let $\mu,\nu$ be finitely
supported probability measures on $\mathbb{R}^{d}$. There exists convex
potentials $u_{\mu}$, $u_{\nu}\in L^{2}\left(  \rho\right)  ,$ normalized in
the sense that
\[
\int u_{\mu}d\rho=\int u_{\nu}d\rho=0,
\]
such that
\[
\left(  \nabla u_{\mu}\right)  _{\#}\rho=\mu,\text{ }\left(  \nabla u_{\nu
}\right)  _{\#}\rho=\nu.
\]
These potentials are uniquely determind in $U$ and their $L^{2}\left(
\rho\right)  $ difference responds linearly to the $W_{2}$-distance,
\[
\left\Vert u_{\mu}-u_{\nu}\right\Vert _{L^{2}\left(  \rho\right)  }
\leqslant\sqrt{1/\kappa}W_{2}\left(  \mu,\nu\right)  .
\]
Importantly, the coefficient is uniform over the support sizes, sites
locations and masses of both targets.
\end{proposition}

\begin{proof}
\textbf{1}. Let $\mu_{0}=\mu$, $\mu_{1}=\nu$ and suppose they are given in the
form
\[
\mu_{0}=\sum_{a=1}^{n_{0}}\alpha_{a}\delta_{z_{a}^{0}},\text{ }\mu_{1}
=\sum_{b=1}^{n_{1}}\beta_{b}\delta_{z_{b}^{1}},\text{ }\alpha_{a},\text{
}\beta_{b}>0,
\]
where $n_{0},n_{1}\geqslant1,$ $\left\{  z_{a}^{0}\right\}  $ and $\left\{
z_{b}^{1}\right\}  $ are respectively distinct support points, $\sum
_{a=1}^{n_{0}}\alpha_{a}=\sum_{b=1}^{n_{1}}\beta_{b}=1$. Let $\pi\in\Pi\left(
\mu_{0},\mu_{1}\right)  $ be an optimal coupling for the quadratic cost. Let
$N=\left\vert \operatorname{supp}\pi\right\vert \geqslant1$ be the number of
distinct mass pairs with positive masses. Then $N\leqslant n_{0}n_{1}.$
Enumerate these coupling pairs as
\[
\left(  y_{i}^{0},y_{i}^{1}\right)  \in\operatorname{supp}\mu_{0}
\times\operatorname{supp}\mu_{1},i=1,...,N
\]
and let $m_{i}=\pi\left(  \left(  y_{i}^{0},y_{i}^{1}\right)  \right)  $. Then
$m_{i}>0$ for all $i$ and $\sum_{i}m_{i}=1.$ Thus
\begin{equation}
\pi=\sum_{i=1}^{N}m_{i}\delta_{\left(  y_{i}^{0},y_{i}^{1}\right)  },\text{
}\mu_{\ell}=\sum_{i}m_{i}\delta_{y_{i}^{\ell}}\text{, }\ell=0,1.
\label{eq:optimal-pi-marginal}
\end{equation}
Let
\[
v_{i}=y_{i}^{1}-y_{i}^{0},\text{ }i=1,...,N.
\]
With $w=W_{2}\left(  \mu,\nu\right)  ,$ we have
\begin{equation}
w^{2}=W_{2}(\mu,\nu)^{2}=\sum_{i}m_{i}|v_{i}|^{2}. \label{eq:finite-velocity}
\end{equation}
Define the displacement interplolation
\[
y_{i}\left(  t\right)  =\left(  1-t\right)  y_{i}^{0}+ty_{i}^{1},\text{
}0\leqslant t\leqslant1.
\]
The optimality implies that different interpolations never meet in the middle
(e.g. \cite[Chapter 5]{villani2009optimal}), that is, for any distinct pairs
$\left(  y_{i}^{0},y_{i}^{1}\right)  ,\left(  y_{j}^{0},y_{j}^{1}\right)  $
with $i\neq j,$
\begin{equation}
\left\vert y_{i}\left(  t\right)  -y_{j}\left(  t\right)  \right\vert >0\text{
for }0<t<1. \label{eq:no-crossing}
\end{equation}
Thus Lemma~\ref{lem:weighted-masses} applies and yields a $C^{1}$ path
$\bar{c}=\left(  \bar{c}_{i}\left(  t\right)  \right)  _{i=1}^{N}:\left(
0,1\right)  \mapsto H_{m}$ satisfying for $0<t<1$,
\begin{equation}
M_{i}\left(  \bar{c}\left(  t\right)  ,y\left(  t\right)  \right)
=m_{i},\text{ }i=1,...,N\text{ and}\sum_{i}m_{i}\bar{c}_{i}\left(  t\right)
=0. \label{eq:fixed-mass-path}
\end{equation}
Here $y(t)=\left(  y_{i}\left(  t\right)  \right)  _{i=1}^{N}\in
\mathcal{Y}_{N}$. Note that Lemma~\ref{lem:weighted-masses} only applies to
the interior of the path where no path crossing occurs, so nothing has been
said about $t=0,1.$ For $0<t<1,$ we define the corresponding cells
\[
P_{i}\left(  t\right)  =P_{i}\left(  \bar{c}\left(  t\right)  ,y\left(
t\right)  \right)  ,\text{ }M_{i}\left(  t\right)  =M_{i}\left(  \bar
{c}\left(  t\right)  ,y\left(  t\right)  \right)  =m_{i}
\]
and their barycenters
\[
B_{i}\left(  t\right)  =\int_{P_{i}\left(  t\right)  }xd\rho\left(  x\right)
,\text{ }b_{i}\left(  t\right)  =\frac{B_{i}\left(  t\right)  }{m_{i}}.
\]
As before, we define the potential associated with these cells as
\[
\bar{u}_{t}\left(  x\right)  =\max_{i}\left\{  x\cdot y_{i}\left(  t\right)
+\bar{c}_{i}\left(  t\right)  \right\}  .
\]
Denote the mean of the potential by
\begin{equation}
a\left(  t\right)  =\int\bar{u}_{t}\left(  x\right)  d\rho\left(  x\right)  ,
\label{eq:mean-ubar}
\end{equation}
which may be non-zero at this stage. To normalize it so that it has zero mean,
we define
\begin{equation}
u_{t}\left(  x\right)  =\bar{u}_{t}\left(  x\right)  -a\left(  t\right)
=\max_{i}\left\{  x\cdot y_{i}\left(  t\right)  +c_{i}\left(  t\right)
\right\}  , \label{eq:utx}
\end{equation}
where
\begin{equation}
c_{i}\left(  t\right)  =\bar{c}_{i}\left(  t\right)  -a\left(  t\right)  .
\label{eq:cbar-shift}
\end{equation}
Hence
\begin{equation}
\int u_{t}\left(  x\right)  d\rho\left(  x\right)  =0.
\label{eq:mean-normalization}
\end{equation}
Write $c\left(  t\right)  =\left(  c_{i}\left(  t\right)  \right)  _{i=1}
^{N},$ $\bar{c}\left(  t\right)  =\left(  \bar{c}_{i}\left(  t\right)
\right)  _{i=1}^{N}.$ Since the definition of $P_{i}\left(  t\right)  ,$
$M_{i}\left(  t\right)  $ only depends on the difference $c_{i}\left(
t\right)  -c_{j}\left(  t\right)  $ and $c_{i}\left(  t\right)  -c_{j}\left(
t\right)  =\bar{c}_{i}\left(  t\right)  -\bar{c}_{j}\left(  t\right)  ,$ we
can replace $\bar{c}\left(  t\right)  $ with $c\left(  t\right)  $ and get
\[
P_{i}(t)=P_{i}(c(t),y(t)),\text{ }M_{i}(c(t),y(t))=m_{i}.
\]

\textbf{2}. Next we show that $c\left(  t\right)  \in C^{1}\left(  \left(
0,1\right)  ,\mathbb{R}^{N}\right)  .$ Fix $t\in\left(  0,1\right)  $ and a
subinterval $J\subset\left(  0,1\right)  $ that contains $t$, the difference
quotient is controlled by a function that is integrable w.r.t. $\rho$,
\begin{equation}
\left\vert \frac{\bar{u}_{t+\epsilon}\left(  x\right)  -\bar{u}_{t}\left(
x\right)  }{\epsilon}\right\vert \leqslant\max_{i}\left\vert v_{i}\right\vert
\left\vert x\right\vert +\max_{i}\sup_{\tau\in J^{\prime}}\left\vert \bar
{c}_{i}^{\prime}\left(  \tau\right)  \right\vert \leqslant C_{J}\left(
1+\left\vert x\right\vert \right)  , \label{eq:ubar-quotient-bound}
\end{equation}
where $J^{\prime}$ is a slightly larger compact subinterval containing $J$ so
that $t+\epsilon\in J^{\prime}$ whenever $\epsilon>0$ is sufficiently small.
The constant $C_{J}$ may depend on $t,$ but note the limit is taken w.r.t.
$\epsilon>0.$ For the above fixed $t,$ let $D_{i}\left(  t\right)
=\operatorname*{int}P_{i}\left(  t\right)  $ and
\[
E_{t}=\mathbb{R}^{d}\backslash\bigcup_{i}D_{i}\left(  t\right)  .
\]
Then $E_{t}$ has zero Lebesgue measure hence zero $\rho$-mass. So for any
$x\in\mathbb{R}^{N}\backslash E_{t}$, there is a unique maximizing index
$i=i\left(  x\right)  $ for $\bar{u}_{t}\left(  x\right)  ,$
\[
\bar{u}_{t}\left(  x\right)  =x\cdot y_{i}\left(  t\right)  +\bar{c}
_{i}\left(  t\right)  >\max_{j:j\neq i}\left\{  x\cdot y_{j}\left(  t\right)
+\bar{c}_{j}\left(  t\right)  \right\}  .
\]
So up to the $\rho$-null set $E_{t}$, the pointwise limit holds,
\[
\lim_{\epsilon\rightarrow0}\frac{\bar{u}_{t+\epsilon}\left(  x\right)
-\bar{u}_{t}\left(  x\right)  }{\epsilon}\rightarrow\sum_{i}1_{P_{i}\left(
t\right)  }\left(  x\right)  \frac{x\cdot\left(  y_{i}\left(  t+\epsilon
\right)  -y_{i}\left(  t\right)  \right)  +\left(  \bar{c}_{i}\left(
t+\epsilon\right)  -\bar{c}_{i}\left(  t\right)  \right)  }{\epsilon}.
\]
Using $\left(  \ref{eq:mean-ubar}\right)  ,$ the difference quotient bound
$\left(  \ref{eq:ubar-quotient-bound}\right)  $ and the dominated convergence
theorem therefore yield
\begin{equation}
a^{\prime}\left(  t\right)  =\lim_{\epsilon\rightarrow0}\int\frac{\bar
{u}_{t+\epsilon}\left(  x\right)  -\bar{u}_{t}\left(  x\right)  }{\epsilon
}d\rho\left(  x\right)  =\sum_{i}\left\{  B_{i}\left(  t\right)  \cdot
v_{i}+m_{i}\bar{c}_{i}^{\prime}\left(  t\right)  \right\}  .
\label{eq:a-derivative}
\end{equation}
The derivative is continuous by Lemma~\ref{lem:weighted-calculus}, i.e.,
$a\left(  t\right)  \in C^{1}\left(  (0,1)\right)  $. Hence the curve
$c\left(  t\right)  $ defined by $\left(  \ref{eq:cbar-shift}\right)  \ $is
$C^{1}.$

\textbf{3}. We now bound the potential difference $\left\Vert u_{s}
-u_{t}\right\Vert _{L^{2}(\rho)}^{2}$ for any $0<s<t<1$ and relate it to our
desired conclusion. For this goal, we first need to compute $\left\Vert
\partial_{t}u_{t}\right\Vert _{L^{2}(\rho)}^{2}.$ Fix $0<t<1$. Define, for
$i\neq j$,
\[
F_{ij}\left(  t\right)  =P_{i}\left(  t\right)  \cap P_{j}\left(  t\right)  ,
\]
\[
g_{ij}\left(  t\right)  =\int_{F_{ij}\left(  t\right)  }\frac{f\left(
x\right)  }{\left\vert y_{i}\left(  t\right)  -y_{j}\left(  t\right)
\right\vert }d\mathcal{H}^{d-1}\left(  x\right)  ,
\]
\[
\beta_{ij}\left(  t\right)  =\int_{F_{ij}\left(  t\right)  }\frac{xf\left(
x\right)  }{\left\vert y_{i}\left(  t\right)  -y_{j}\left(  t\right)
\right\vert }d\mathcal{H}^{d-1}\left(  x\right)  .
\]
These quantities are well-defined by the no-crossing property $\left(
\ref{eq:no-crossing}\right)  $ and the assumptions of $\rho.$ To simplify
\textbf{notations}, in the following proof we write $P_{i},$ $M_{i},$ $B_{i},$
$b_{i},$ $\Sigma_{i}$ without explicitly referring to the dependence on $t,$
so for example, the differentiation $Db_{i}$ means $Db_{i}\left(  c\left(
t\right)  ,y\left(  t\right)  \right)  .$ We also define
\[
r_{i}=c_{i}^{\prime}\left(  t\right)  ,\text{ }r=\left(  r_{i}\left(
t\right)  \right)  _{i=1}^{N},\text{ }v=\left(  v_{i}\right)  _{i=1}^{N}
\]
and
\[
k_{i}=r_{i}+b_{i}\cdot v_{i}.
\]
For any $t\in(0,1),$ $x\in P_{i}\backslash E_{t}$. Repeating the difference
quotient procedure that has been carried out on $\bar{u}_{t},$ or simply
differentiating $\left(  \ref{eq:utx}\right)  $ w.r.t. $t$ and using $\left(
\ref{eq:a-derivative}\right)  ,$ we have
\begin{equation}
\partial_{t}u_{t}\left(  x\right)  =x\cdot v_{i}+r_{i}=\left(  x-b_{i}\right)
\cdot v_{i}+k_{i}. \label{eq:partial-differential-ut}
\end{equation}
Hence the dominated convergence theorem yields,
\begin{equation}
0=\int\partial_{t}u_{t}d\rho=\sum_{i}\left(  B_{i}\cdot v_{i}+m_{i}
r_{i}\right)  =\sum_{i}m_{i}k_{i}. \label{eq:centered-velocity}
\end{equation}
The last equality uses $B_{i}=m_{i}b_{i}.$ Differentiating $\left(
\ref{eq:fixed-mass-path}\right)  $ w.r.t. $t\in(0,1)$ we get
\[
0=\partial_{t}M_{i}\left(  c\left(  t\right)  ,y\left(  t\right)  \right)
=DM_{i}\left(  c\left(  t\right)  ,y\left(  t\right)  \right)  \left[
c^{\prime}\left(  t\right)  ,y^{\prime}\left(  t\right)  \right]
=DM_{i}\left[  r,v\right]  .
\]
Using Lemma \ref{lem:weighted-calculus} and the notations just introduced, we
have, for any $q\in\mathbb{R}^{N}$,
\begin{align*}
0  &  =\sum_{i}q_{i}DM_{i}\left[  r,v\right]  =\sum_{i}\sum_{j\neq i}
q_{i}\left\{  g_{ij}\left(  r_{i}-r_{j}\right)  +\beta_{ij}\cdot\left(
v_{i}-v_{j}\right)  \right\} \\
&  =\sum_{i<j}\left(  q_{i}-q_{j}\right)  \left\{  g_{ij}\left(  r_{i}
-r_{j}\right)  +\beta_{ij}\cdot\left(  v_{i}-v_{j}\right)  \right\} \\
&  =\sum_{i}r_{i}\sum_{j\neq i}g_{ij}\left(  q_{i}-q_{j}\right)  +\sum
_{i}v_{i}\cdot\sum_{j\neq i}\beta_{ij}\left(  q_{i}-q_{j}\right) \\
&  =\sum_{i}r_{i}DM_{i}\left[  q,0\right]  +\sum_{i}v_{i}\cdot DB_{i}\left[
q,0\right]  .
\end{align*}
Since $B_{i}=m_{i}b_{i}=M_{i}b_{i}$, the equation continues
\begin{align*}
0  &  =\sum_{i}r_{i}DM_{i}\left[  q,0\right]  +\sum_{i}v_{i}\cdot\left\{
b_{i}DM_{i}\left[  q,0\right]  +m_{i}Db_{i}\left[  q,0\right]  \right\} \\
&  =\sum_{i}\left(  r_{i}+b_{i}\cdot v_{i}\right)  DM_{i}\left[  q,0\right]
+\sum_{i}m_{i}v_{i}\cdot Db_{i}\left[  q,0\right] \\
&  =\sum_{i}k_{i}DM_{i}\left[  q,0\right]  +\sum_{i}m_{i}v_{i}\cdot
Db_{i}\left[  q,0\right]  .
\end{align*}
We therefore obtain
\begin{equation}
\sum_{i}k_{i}DM_{i}\left[  q,0\right]  =-\sum_{i}m_{i}v_{i}\cdot Db_{i}\left[
q,0\right]  . \label{eq:center-cancellation}
\end{equation}
By $\left(  \ref{eq:centered-velocity}\right)  $ and the range assertion of
Lemma~\ref{lem:weighted-calculus}, we can fix a vector $q\in\mathbb{R}^{N}$
such that
\begin{equation}
DM_{i}\left[  q,0\right]  =m_{i}k_{i}. \label{eq:solve-test-direciton}
\end{equation}
By the defintion of $H$ given in Lemma \ref{lem:joint}, $\left(
\ref{eq:center-cancellation}\right)  $ and $\left(
\ref{eq:solve-test-direciton}\right)  ,$ we obtain that
\[
H\left(  q,v\right)  =\sum_{i}m_{i}\dot{b}_{i}\cdot v_{i}=-\sum_{i}k_{i}
DM_{i}[q,0]=-\sum_{i}m_{i}k_{i}^{2}.
\]
Since $H$ is linear in $v,$
\[
H\left(  q,-v\right)  =-H\left(  q,v\right)  =\sum_{i}m_{i}k_{i}^{2}.
\]
With the defintion of $I$ given in Lemma \ref{lem:joint}, we also have
\[
I(q)=\sum_{i}M_{i}\left(  \dot{M}_{i}/M_{i}\right)  ^{2}=\sum_{i}m_{i}
k_{i}^{2},
\]
and the $D$ in Lemma \ref{lem:joint},
\[
D\left(  v\right)  =\sum_{i}M_{i}v_{i}^{\top}\left(  \kappa^{-1}I_{d}
-\Sigma_{i}\right)  \cdot v_{i}=\kappa^{-1}\sum_{i}m_{i}\left\vert
v_{i}\right\vert ^{2}-\sum_{i}m_{i}v_{i}^{\top}\Sigma_{i}\cdot v_{i}.
\]
Therefore Invoking Lema \ref{lem:joint} with $\xi=-v$, we get
\[
\sum_{i}m_{i}k_{i}^{2}=2H\left(  q,-v\right)  -I\left(  q\right)  \leqslant
D\left(  -v\right)  =\kappa^{-1}\sum_{i}m_{i}\left\vert v_{i}\right\vert
^{2}-\sum_{i}m_{i}v_{i}^{\top}\Sigma_{i}\cdot v_{i}.
\]
Note that
\[
\int_{P_{i}}\left\vert \left(  x-b_{i}\right)  \cdot v_{i}\right\vert
^{2}d\rho=v_{i}^{\top}\left[  \int_{P_{i}}\left(  x-b_{i}\right)  \left(
x-b_{i}\right)  ^{\top}d\rho\right]  \cdot v_{i}=m_{i}v_{i}^{\top}\Sigma
_{i}\cdot v_{i}.
\]
Combing the above equations, we obtain the total variacen estimate
\begin{equation}
\sum_{i}\int_{P_{i}}\left\vert \left(  x-b_{i}\right)  \cdot v_{i}\right\vert
^{2}d\rho+\sum_{i}m_{i}k_{i}^{2}\leqslant\kappa^{-1}\sum_{i}m_{i}\left\vert
v_{i}\right\vert ^{2}. \label{eq:total-variance}
\end{equation}
For each cell, $b_{i}$ and $v_{i}$ are constant with respect to $x$, so
\[
\int_{P_{i}}\left(  x-b_{i}\right)  \cdot v_{i}d\rho=B_{i}\cdot v_{i}
-m_{i}b_{i}\cdot v_{i}=0.
\]
Expanding the square in $\left(  \ref{eq:partial-differential-ut}\right)  $
and combing the above equations, we therefore get
\begin{align*}
\left\Vert \partial_{t}u_{t}\right\Vert _{L^{2}(\rho)}^{2}  &  =\sum
_{i}\left\{  \int_{P_{i}}\left\vert \left(  x-b_{i}\right)  \cdot
v_{i}\right\vert ^{2}d\rho+2k_{i}\left(  B_{i}-m_{i}b_{i}\right)  \cdot
v_{i}+m_{i}k_{i}^{2}\right\} \\
&  =\sum_{i}\int_{P_{i}}\left\vert \left(  x-b_{i}\right)  \cdot
v_{i}\right\vert ^{2}d\rho+\sum_{i}m_{i}k_{i}^{2}\leqslant\kappa^{-1}\sum
_{i}m_{i}\left\vert v_{i}\right\vert ^{2}=\frac{w^{2}}{\kappa}.
\end{align*}
See Remark \ref{rmk:measurability} for details on the measurability of
$\partial_{t}u_{t}$. For $0<s<t<1$, the Jansen's inequality w.r.t. the $\tau
$-integral implies that
\begin{align}
\left\Vert u_{t}-u_{s}\right\Vert _{L^{2}\left(  \rho\right)  }^{2}  &
=\int\left\vert \int_{s}^{t}\partial_{\tau}u_{\tau}\left(  x\right)
d\tau\right\vert ^{2}d\rho\left(  x\right)  =\left(  t-s\right)  ^{2}
\int\left\vert \frac{1}{t-s}\int_{s}^{t}\partial_{\tau}u_{\tau}\left(
x\right)  d\tau\right\vert ^{2}d\rho\left(  x\right)
\label{eq:interval-norm2}\\
&  \leqslant(t-s)\int_{s}^{t}\Vert\partial_{\tau}u_{\tau}\Vert_{L^{2}(\rho
)}^{2}d\tau\leqslant\frac{w^{2}}{\kappa}(t-s)^{2}.\nonumber
\end{align}

\textbf{4}. It now remains to consider the extreme points $t=0$ and $t=1.$ In
the remaining proof $\ell$ always denotes either $0$ or $1$, and $\mu_{\ell}$
is defined by $\left(  \ref{eq:utx}\right)  $. By H\"{o}lder's inequality,
\[
\left\vert b_{i}\left(  t\right)  \right\vert \leqslant\frac{1}{m_{i}}
\int_{P_{i}}|x|d\rho\leqslant\left(  \frac{1}{m_{i}}\int_{P_{i}}|x|^{2}
d\rho\right)  ^{1/2}\leqslant\left(  \frac{1}{m_{i}}\int|x|^{2}d\rho\right)
^{1/2}.
\]
By $\left(  \ref{eq:total-variance}\right)  ,$ $m_{i}k_{i}^{2}\leqslant
w^{2}/\kappa.$ Hence
\[
\left\vert c_{i}^{\prime}\left(  t\right)  \right\vert =\left\vert k_{i}
-b_{i}\cdot v_{i}\right\vert \leqslant\frac{w}{\sqrt{\kappa m_{i}}}+\left\vert
v_{i}\right\vert \left(  \frac{1}{m_{i}}\int|x|^{2}d\rho\right)
^{1/2}\triangleq K_{i}.
\]
Thus
\[
\left\vert c_{i}\left(  t\right)  -c_{i}\left(  s\right)  \right\vert
\leqslant K_{i}\left\vert t-s\right\vert
\]
and each $c_{i}$ extends continuously to $\left[  0,1\right]  $. Sending
$t\rightarrow\ell$ in $\left(  \ref{eq:utx}\right)  $ shows that $u_{0},u_{1}$
are finite and convex. Moreover
\[
\left\vert u_{t}\left(  x\right)  -u_{\ell}\left(  x\right)  \right\vert
\leqslant\max_{i}\left\{  \left\vert x\right\vert \left\vert y_{i}\left(
t\right)  -y_{i}\left(  \ell\right)  \right\vert +\left\vert c_{i}\left(
t\right)  -c_{i}\left(  \ell\right)  \right\vert \right\}  \rightarrow0\text{
as }t\rightarrow\ell,\ \ell=0,1
\]
Thus the path $u_{t}$ extends continuously to $\left[  0,1\right]  $ for $x$
in any compact subsets of $\mathbb{R}^{d}$. Since $y_{i}$ and $c_{i}$ are
bounded on $\left[  0,1\right]  $, so $\left\vert u_{t}\left(  x\right)
\right\vert \leqslant C\left(  1+\left\vert x\right\vert \right)  $ for some
$C>0,$ thus the dominated convergence theorem yields
\begin{equation}
\left\Vert u_{t}-u_{\ell}\right\Vert _{L^{2}(\rho)}^{2}\rightarrow0\text{ as
}t\rightarrow\ell,\ \ell=0,1. \label{eq:boundary-norm2}
\end{equation}
In particular $u_{\ell}\in L^{2}(\rho),$ $\int u_{\ell}d\rho=0$.

\textbf{5}. We finally verify that the convex function $u_{\ell},$ $\ell=0,1$
give the optimal maps. As $t\rightarrow\ell,$ some of the paths in the list
$\left\{  y_{i}\left(  t\right)  \right\}  $ will collapse. Let
\[
Z_{\ell}=\left\{  y_{i}^{\ell}:1\leqslant i\leqslant N\right\}
\]
enumerate the \textit{distinct} sites that these collapsing paths will create
as $t\rightarrow\ell$. Choose the largest $c_{i}\left(  \ell\right)  $
corresponding to a given site $z\in Z_{\ell},$
\[
C_{\ell}\left(  z\right)  =\max_{i:y_{i}^{\ell}=z}c_{i}\left(  \ell\right)  .
\]
Thus
\[
u_{\ell}(x)=\max_{z\in Z_{\ell}}\left\{  x\cdot z+C_{\ell}\left(  z\right)
\right\}  .
\]
Let $\mathcal{E}_{\ell}$ be the union of the intersections
\[
\bigcup_{z,z^{\prime}\in Z_{\ell}:z\neq z^{\prime}}\left\{  x:x\cdot
z+C_{\ell}\left(  z\right)  =x\cdot z^{\prime}+C_{\ell}\left(  z^{\prime
}\right)  \right\}  .
\]
It is a $\rho$-null set. If $Z_{\ell}$ contains only one element,
$\mathcal{E}_{\ell}$ is empty. Let
\[
L_{i}\left(  t,x\right)  =x\cdot y_{i}\left(  t\right)  +c_{i}\left(
t\right)  ,\text{ }i=1,...,N.
\]
Denote by
\[
I_{\ell}\left(  x\right)  =\left\{  i:L_{i}\left(  \ell,x\right)  =u_{\ell
}\left(  x\right)  \right\}
\]
the set of maximizing indices. For $x\notin\mathcal{E}_{\ell},$ $u_{\ell}(x)$
has a unique maximizing slope among $Z_{\ell}$, denoted by $z_{\ell}(x)$. This
does not mean that there is a unique maximizing index in $I_{\ell}(x),$
several different indices in $I_{\ell}(x)$ might give the same $u_{\ell}(x).$
But since $x\notin\mathcal{E}_{\ell},$ $z_{\ell}(x)$ is unqiue and the strict
domination holds
\begin{equation}
u_{\ell}(x)=x\cdot z_{\ell}(x)+C_{\ell}\left(  z_{\ell}(x)\right)
>\max_{i\notin I_{\ell}\left(  x\right)  }L_{i}\left(  \ell,x\right)  .
\label{eq:strict-domination}
\end{equation}
Since $L_{i}\left(  \ell,x\right)  ,$ $i=1,...,N$ are continuous in $x,$ small
perturbations in $x$ will not change the strict domination. In other words,
when $t$ approaches $\ell$, $z_{\ell}(x)$ remains constant around some
neighbourhood of $x,$ hence, at this $x$,
\[
\nabla u_{\ell}\left(  x\right)  =z_{\ell}\left(  x\right)  .
\]
Analogously, Let $\mathcal{E}_{t}$ be the union of the intersections at $t,$
\[
\bigcup_{i\neq j}\left\{  x:L_{i}\left(  t,x\right)  =L_{j}\left(  t,x\right)
\right\}  .
\]
Let $t_{n}\rightarrow\ell$ as $n\rightarrow\infty.$ The countable union
$\left(  \cup\mathcal{E}_{t_{n}}\right)  \cup\mathcal{E}_{\ell}$ has zero
$\rho$-mass. For any $x$ outside this union, the same strict domination
$\left(  \ref{eq:strict-domination}\right)  $ continues to hold. Since there
is no crossing $\left(  \ref{eq:no-crossing}\right)  $ for $t_{n}\in\left(
0,1\right)  ,$ stronger strict domination can be said: there is a unique
maximizing index $i\left(  t_{n},x\right)  $ for $u_{t_{n}}(x)$ (recall
$\left(  \ref{eq:utx}\right)  $) such that, for all large $n,$
\begin{equation}
u_{t_{n}}(x)=x\cdot y_{i\left(  t_{n},x\right)  }(t_{n})+c_{i\left(
t_{n},x\right)  }\left(  t_{n}\right)  >\max_{i\neq i\left(  t_{n},x\right)
}L_{i}\left(  t_{n},x\right)  . \label{eq:strict-domination-tn}
\end{equation}
Since we already show that $y_{i}\left(  t\right)  ,$ $c_{i}\left(  t\right)
$ are continuous on $\left[  0,1\right]  $ and the functions $L_{i}\left(
t,x\right)  ,$ $i=1,...,N$ are continuous in $\left(  t,x\right)  .$ The two
sides of $\left(  \ref{eq:strict-domination-tn}\right)  $ converge to those of
$\left(  \ref{eq:strict-domination}\right)  $. So small perturbations in
$\left(  t,x\right)  $ will not change the strict domination. In other words,
when $t_{n}$ stays close to $\ell,$ $y_{i\left(  t_{n},x\right)  }(t_{n})$
will remain constant in some small neighbourhood of $x$. It follows that
\begin{equation}
\nabla u_{t_{n}}\left(  x\right)  =y_{i\left(  t_{n},x\right)  }
(t_{n})\rightarrow z_{\ell}\left(  x\right)  =\nabla u_{\ell}\left(  x\right)
\text{ as }n\rightarrow\infty. \label{eq:grad-convergence}
\end{equation}
Note that $i\left(  t_{n},x\right)  $ might not converge, instead it might
jump between different values in $I_{\ell}\left(  x\right)  .$ But $I_{\ell
}\left(  x\right)  $ is a finite set, so convergent subsequence of $i\left(
t_{n},x\right)  $ does exist. Since $\left(  \ref{eq:grad-convergence}\right)
$ holds for all $x$ outside $\left(  \cup\mathcal{E}_{t_{n}}\right)
\cup\mathcal{E}_{\ell},$ we obtain by the dominated convergence theorem that,
for any $\varphi\in C_{b}\left(  \mathbb{R}^{d}\right)  $,
\[
\int\varphi(\nabla u_{\ell})d\rho=\lim_{n\rightarrow\infty}\int\varphi\left(
\nabla u_{t_{n}}\right)  d\rho=\lim_{n\rightarrow\infty}\sum_{i}m_{i}
\varphi(y_{i}(t_{n}))=\sum_{i}m_{i}\varphi(y_{i}^{\ell})=\int\varphi
d\mu_{\ell}.
\]
Therefore we have proved that
\[
(\nabla u_{\ell})_{\#}\rho=\mu_{\ell},\text{ }\ell=0,1.
\]
The triangle inequality and $\left(  \ref{eq:boundary-norm2}\right)  $ yield
\[
\left\vert \left\Vert u_{t}-u_{s}\right\Vert _{L^{2}\left(  \rho\right)
}-\left\Vert u_{1}-u_{0}\right\Vert _{L^{2}\left(  \rho\right)  }\right\vert
\leqslant\left\Vert u_{t}-u_{1}\right\Vert _{L^{2}\left(  \rho\right)
}+\left\Vert u_{s}-u_{0}\right\Vert _{L^{2}\left(  \rho\right)  }\rightarrow0
\]
as $s\rightarrow0$, $t\rightarrow1$. Hence by $\left(  \ref{eq:interval-norm2}
\right)  $ we get
\[
\left\Vert u_{1}-u_{0}\right\Vert _{L^{2}\left(  \rho\right)  }\leqslant
\sqrt{1/\kappa}w=\sqrt{1/\kappa}W_{2}\left(  \mu,\nu\right)  .
\]
The desired potentials are $u_{\mu}=u_{0},$ $u_{\nu}=u_{1}.$

\textbf{6}. Finally we prove the uniqueness of the normalized potentials. We
only prove the uniqueness of $u_{\mu},$ the other follows analogously. If $u$
and $\bar{u}$ are convex potentials in $L^{2}(\rho)$ with zero mean w.r.t.
$\rho$ and $\left(  \nabla u\right)  _{\#}\rho=\left(  \nabla\bar{u}\right)
_{\#}\rho=\mu$, Brenier's theorem \cite[Theorem 2.32]{villani2003topics} tells
us that $\nabla u=\nabla\bar{u},$ $\rho$-a.e.. Because $f>0$ on every compact
set in $U$, the same equality holds for almost all points in $U$ under
Lebesgue measure. Hence the locally Lipschitz function $h=u-\bar{u}$ has zero
weak gradient \cite[Section 5.8.3 Theorem 5]{evans2022partial}. Let
$\eta_{\epsilon}$ be a standard mollifier supported in $\bar{B}\left(
0,\epsilon\right)  $ with $\epsilon>0$. Then in every open set compactly
contained in $U$, the convolution $h\ast\eta_{\epsilon}$ is smooth there for
sufficiently small $\epsilon$ and has zero gradient there. Thus for any $x,$
$y\in U$,
\[
\left(  h\ast\eta_{\epsilon}\right)  \left(  y\right)  -\left(  h\ast
\eta_{\epsilon}\right)  (x)=\int_{0}^{1}\nabla\left(  h\ast\eta_{\epsilon
}\right)  \left(  x+\theta\left(  y-x\right)  \right)  \cdot\left(
y-x\right)  d\theta=0.
\]
So $h\ast\eta_{\epsilon}$ is constant for all $\epsilon>0.$ Sending
$\epsilon\rightarrow0$ yields that $h$ is constant in $U$. Using the zero mean
condition, we conclude that the constant must be zero, thus $u=\bar{u}$ in
U$.$ The proof is complete.
\end{proof}

\begin{remark}
\label{rmk:measurability}In the proof of Proposition \ref{prop:finite-strong},
we integrated $\partial_{t}u_{t}\left(  x\right)  $ against $dtd\rho.$ The
measurability of this function needs verification. Let
\[
L_{i}\left(  t,x\right)  =x\cdot y_{i}\left(  t\right)  +c_{i}\left(
t\right)  ,\text{ }i=1,...,N.
\]
Define the region where $L_{i}$ strictly dominates the remaining,
\[
D_{i}=\left\{  \left(  t,x\right)  \in\left(  0,1\right)  \times\mathbb{R}
^{d}:L_{i}\left(  t,x\right)  >L_{j}\left(  t,x\right)  \text{ for all }j\neq
i\right\}  .
\]
Their union constitue the region where a unique plane strictly dominates
others, also define
\[
\mathcal{E}=\left(  \left(  0,1\right)  \times\mathbb{R}^{d}\right)
\backslash\bigcup_{i}D_{i}.
\]
Since $L_{i}\left(  t,x\right)  $ is conitnuous in $\left(  t,x\right)  $,
$D_{i}$ is an open subset of $\left(  0,1\right)  \times\mathbb{R}^{d}$. The
union $\bigcup_{i}D_{i}$ is also open. Thus $\mathcal{E}$ is a Borel subset of
$\left(  0,1\right)  \times\mathbb{R}^{d}$. Let $\lambda$ be the Lebesgue
measure. Equip $\left(  0,1\right)  \times\mathbb{R}^{d}$ with the product
measure $\lambda\times\rho$ and the usual Borel $\sigma$-algebra. We claim
that $\mathcal{E}$ has zero $\left(  \lambda\times\rho\right)  $-measure$.$
Indeed, being Borel, the set $\mathcal{E}$ is $\left(  \lambda\times
\rho\right)  $-measurable. For each $t\in\left(  0,1\right)  ,$ the slice
$\mathcal{E}_{t}=\left\{  x\in\mathbb{R}^{d}:\left(  t,x\right)
\in\mathcal{E}\right\}  $ has zero Lebesgue measure (it is contained in the
set where at least two planes meet and dominate the remaining), hence zero
$\rho$-mass, i.e. $\rho\left(  \mathcal{E}_{t}\right)  =0.$ Therefore, by
Tonelli's theorem,
\[
\int_{\left(  0,1\right)  \times\mathbb{R}^{d}}1_{\mathcal{E}}\left(
t,x\right)  dtdx=\int_{0}^{1}\rho\left(  \mathcal{E}_{t}\right)  dt=0,
\]
showing that $\mathcal{E}$ has zero $\left(  \lambda\times\rho\right)
$-measure. For any $\left(  t,x\right)  $ outside $\mathcal{E}$, there is a
unique maximizing index for $u_{t}(x),$ thus at this $\left(  t,x\right)  $,
$u_{t}(x)$ equals a single linear function hence $\partial_{t}u_{t}(x)$ is
well-defined. In fact $\left(  \ref{eq:ubar-quotient-bound}\right)  $ already
shows that $\bar{u}_{t}(x)$ is locally Lipschitz on $\left(  0,1\right)  ,$
thus $u_{t}(x)=\bar{u}_{t}(x)-a\left(  t\right)  $ is absolutely continuous
w.r.t. $t$ since $a\left(  t\right)  \in C^{1}\left(  \left(  0,1\right)
\right)  .$ Combining these with the just proved fact that $\mathcal{E}$ has
zero $\left(  \lambda\times\rho\right)  $-measure, we see that $\left(
t,x\right)  \mapsto\partial_{t}u_{t}(x)$ is $\left(  \lambda\times\rho\right)
$-measurable. By Fubini's theorem we also have a $\rho$-null set $\mathcal{N}$
such that $\lambda\left(  \mathcal{E}_{x}\right)  =0$ whenever $x\notin
\mathcal{N}$, where $\mathcal{E}_{x}=\left\{  t\in\left(  0,1\right)  :\left(
t,x\right)  \in\mathcal{E}\right\}  .$ This says $\partial_{t}u_{t}\left(
x\right)  $ is well-defined for all $x\notin\mathcal{N}$ and almost all
$t\in\left(  0,1\right)  .$ The point is that $\mathcal{N}$ is independent of
$t.$ This therefore justifies the integration $\int\left\vert \int_{s}
^{t}\partial_{\tau}u_{\tau}\left(  x\right)  d\tau\right\vert ^{2}d\rho$ for
$0<s<t<1.$
\end{remark}

\subsection{Probability interpretation}

The proof of Proposition \ref{prop:finite-strong} has a natural probability
interpretation. Let $X\sim\rho.$ Define index random variable that picks out
the dominating plane,
\[
I=\min\left\{  i:X\in P_{i}\right\}  .
\]
The minimum specifies the index on the intersection of planes. Then using our
earlier notations, $\mathbb{\rho}(I=i)=m_{i}$ and $\sum_{i}m_{i}=1$. The
conditional mean and covariance are respectively
\[
b_{i}=\mathbb{E}_{\rho}\left(  X\mid I=i\right)  ,\text{ }\Sigma
_{i}=\mathbb{E}_{\rho}\left[  \left(  X-b_{i}\right)  \left(  X-b_{i}\right)
^{\top}\mid I=i\right]  .
\]
For the potentail $u_{t}\left(  x\right)  =\max_{i}\left\{  x\cdot
y_{i}\left(  t\right)  +c_{i}\left(  t\right)  \right\}  $ as in $\left(
\ref{eq:utx}\right)  ,$ differentiation w.r.t. $t$ suggests we introduce the
random variable
\[
Y=\partial_{t}u_{t}\left(  X\right)  =r_{I}+v_{I}\cdot X,
\]
This holds almost surely under $\rho$. Also $Y\in L^{2}\left(  \rho\right)  $.
By direct computations we get with the definition of $k_{i}$ in the proof,
\[
\mathbb{E}_{\rho}\left[  Y\mid I=i\right]  =\frac{1}{m_{i}}\int_{P_{i}}\left(
r_{i}+v_{i}\cdot x\right)  d\rho\left(  x\right)  =r_{i}+v_{i}\cdot
b_{i}=k_{i}.
\]
Together with $\left(  \ref{eq:centered-velocity}\right)  ,$ this yields that
$\mathbb{E}_{\rho}Y=\sum_{i}m_{i}k_{i}=0,$ hence
\[
\operatorname{Var}_{\rho}\left[  \mathbb{E}_{\rho}\left(  Y\mid I=i\right)
\right]  =\sum_{i}m_{i}k_{i}^{2}.
\]
Moreover sutstituting $k_{i}$ we get the conditional variance,
\begin{align*}
\operatorname{Var}_{\rho}\left[  Y\mid I=i\right]   &  =\frac{1}{m_{i}}
\int_{P_{i}}\left(  r_{i}+v_{i}\cdot x-k_{i}\right)  ^{2}d\rho\left(  x\right)
\\
&  =\frac{1}{m_{i}}\int_{P_{i}}\left(  \left(  x-b_{i}\right)  \cdot
v_{i}\right)  ^{2}d\rho\left(  x\right)  =v_{i}^{\top}\Sigma_{i}\cdot v_{i}.
\end{align*}
So
\[
\mathbb{E}_{\rho}\left[  \operatorname{Var}_{\rho}\left(  Y\mid I=i\right)
\right]  =\sum_{i}m_{i}v_{i}^{\top}\Sigma_{i}\cdot v_{i}.
\]
Therefore the central piece in the proof of Proposition
\ref{prop:finite-strong} is the total variance formula
\[
\left\Vert \partial_{t}u_{t}\right\Vert _{L^{2}(\rho)}^{2}=\operatorname{Var}
_{\rho}\left(  Y\right)  =\mathbb{E}_{\rho}\left[  \operatorname{Var}_{\rho
}\left(  Y\mid I=i\right)  \right]  +\operatorname{Var}_{\rho}\left[
\mathbb{E}_{\rho}\left(  Y\mid I=i\right)  \right]  \leqslant\frac{w^{2}
}{\kappa}.
\]

\subsection{Proof of the main results}

\begin{proof}
[Proof of Theorem~\ref{thm:strong-main}]\textbf{1}. Fix the source in $\left(
\ref{eq:strong-source}\right)  $. We first construct the potential for an
arbitrary $\mu\in\mathcal{P}_{2}(\mathbb{R}^{d})$. For each $n\in\mathbb{N}$,
write $y=\left(  y_{1},\ldots,y_{d}\right)  $ and define the map
\[
U_{n}\left(  y\right)  =\left\{
\begin{array}
[c]{ll}
n^{-1}\left(  \left\lfloor ny_{1}\right\rfloor ,\ldots,\left\lfloor
ny_{d}\right\rfloor \right)  , & \left\vert y\right\vert \leqslant n,\\
0, & \left\vert y\right\vert >n.
\end{array}
\right.
\]
Let $\mu_{n}=(U_{n})_{\#}\mu.$ For $\left\vert y\right\vert \leqslant n,$
\[
\left\lfloor ny_{i}\right\rfloor \leqslant ny_{i}\leqslant\left\lfloor
ny_{i}\right\rfloor +1,\text{ }\forall i.
\]
Hence $\left\vert n^{-1}\left\lfloor ny_{i}\right\rfloor -y_{i}\right\vert
\leqslant1/n.$ The coupling $\left(  \operatorname*{Id},U_{n}\right)  _{\#}
\mu$ therefore yields
\begin{equation}
W_{2}\left(  \mu_{n},\mu\right)  ^{2}\leqslant\int\left\vert U_{n}\left(
y\right)  -y\right\vert ^{2}d\mu\left(  y\right)  \leqslant\frac{d}{n^{2}
}+\int_{\left\{  \left\vert y\right\vert >n\right\}  }\left\vert y\right\vert
^{2}d\mu\left(  y\right)  \rightarrow0. \label{eq:quantization}
\end{equation}
Let $u_{n}$ be the potentials from Proposition~\ref{prop:finite-strong} such
that $\int u_{n}d\rho=0$. The triangle inequality implies, as $m,$
$n\rightarrow\infty,$
\[
\left\Vert u_{n}-u_{m}\right\Vert _{L^{2}\left(  \rho\right)  }\leqslant
\sqrt{1/\kappa}\left\vert W_{2}\left(  \mu_{n},\mu\right)  +W_{2}\left(
\mu_{m},\mu\right)  \right\vert \rightarrow0.
\]
So $\left\{  u_{n}\right\}  $ is Cauchy in $L^{2}\left(  \rho\right)  $ and
converges to some $u_{\infty}$ in $L^{2}(\rho)$. By H\"{o}lder's inequality,
as $m,$ $n\rightarrow\infty,$
\[
\left\vert \int u_{\infty}d\rho-\int u_{n}d\rho\right\vert \leqslant\left\Vert
u_{n}-u_{m}\right\Vert _{L^{2}\left(  \rho\right)  }\rightarrow0.
\]
Hence $\int u_{\infty}d\rho=0.$

\textbf{2}. We next show that $u_{\infty}$ is finite convex and induces the
optimal map $T_{\mu}$. For the Dirac target $\delta_{0},$ the normalized
potential is identically zero. Hence applying
Proposition~\ref{prop:finite-strong} to $\mu_{n}$ and $\delta_{0}\ $gives
\[
\left\Vert u_{n}\right\Vert _{L^{2}\left(  \rho\right)  }\leqslant
\sqrt{1/\kappa}W_{2}\left(  \mu_{n},\delta_{0}\right)  \leqslant W_{2}\left(
\mu_{n},\mu\right)  +W_{2}\left(  \mu,\delta_{0}\right)  <\infty.
\]
This together with $\left(  \ref{eq:quantization}\right)  $ shows that
\[
K\triangleq\sup_{n}\left\Vert u_{n}\right\Vert _{L^{2}\left(  \rho\right)
}<\infty.
\]
For any closed ball $\bar{B}\subset U$, let $f_{\bar{B}}=\min_{\overline{B}
}f>0$. Then
\[
\int_{\bar{B}}\left\vert u_{n}\right\vert dx\leqslant f_{\bar{B}}^{-1}
\int\left\vert u_{n}\right\vert d\rho\leqslant f_{\bar{B}}^{-1}\left\Vert
u_{n}\right\Vert _{L^{2}\left(  \rho\right)  }\leqslant f_{\bar{B}}
^{-1}K<\infty.
\]
By Lemma~\ref{lem:convex-compactness}, up to a subsequence, $u_{n}$ converges
locally uniformly to some finite convex function $u_{\mu}$. Hence we have
\[
\int_{\bar{B}}\left\vert u_{n}-u_{\mu}\right\vert f\left(  x\right)
dx\rightarrow0,
\]
and
\[
\int_{\bar{B}}\left\vert u_{n}-u_{\infty}\right\vert f\left(  x\right)
dx\leqslant\rho\left(  \bar{B}\right)  ^{1/2}\left(  \int_{\bar{B}}\left\vert
u_{n}-u_{\infty}\right\vert ^{2}f\left(  x\right)  dx\right)  ^{1/2}
\rightarrow0.
\]
It follows that $u_{\mu}$ equals $u_{\infty},$ $\rho$-a.e. in $\bar{B}.$ Since
$\bar{B}$ is an arbitrary subset of $U$ and $\rho\left(  U\right)  =1,$ so
$u_{\mu}$ equals $u_{\infty},$ $\rho$-a.e$.$ Thus $u_{\mu}\in L^{2}\left(
\rho\right)  $ and $\int u_{\mu}d\rho=0$.

\textbf{3}. By Lemma~\ref{lem:convex-compactness}, $\nabla u_{n}$ converges to
$\nabla u_{\mu}$, $\rho$-almost everywhere. By the metric property of
Wasserstein distance \cite[Theorem 7.12]{villani2003topics} , $\left(
\ref{eq:quantization}\right)  $ implies that $\mu_{n}\rightarrow\mu$ weakly.
Therefore, for $a\in C_{b}\left(  \mathbb{R}^{d}\right)  $,
\begin{equation}
\int a\left(  \nabla u_{\mu}\right)  d\rho=\lim_{n\rightarrow\infty}\int
a\left(  \nabla u_{n}\right)  d\rho=\lim_{n\rightarrow\infty}\int ad\mu
_{n}=\int ad\mu. \label{eq:limit-gradient-law}
\end{equation}
In particular $\left(  \nabla u_{\mu}\right)  _{\#}\rho=\mu$ and
\[
\int|\nabla u_{\mu}|^{2}d\rho=\int|y|^{2}d\mu<\infty.
\]
Brenier's theorem \cite[Theorem 2.32]{villani2003topics} then implies that
$\nabla u_{\mu}=T_{\mu}$. The same lines as in Proposition
\ref{prop:finite-strong} yields the uniqueness of such potential $u_{\mu}.$
The uniqueness also implies the whole sequence $u_{n},$ not just a
subsequence, converges locally uniformly and in $L^{2}\left(  \rho\right)  .$
Apply the same construction to $\nu$, with finite approximations $\nu_{n}$ and
potentials $v_{n}$ converging to some finite convex $u_{\nu}$ in $L^{2}(\rho
)$. Then
\[
\Vert u_{\mu}-u_{\nu}\Vert_{L^{2}(\rho)}=\lim_{n\rightarrow\infty}\Vert
u_{n}-v_{n}\Vert_{L^{2}(\rho)}\leqslant\sqrt{1/\kappa}\lim_{n\rightarrow
\infty}W_{2}(\mu_{n},\nu_{n})=\sqrt{1/\kappa}W_{2}(\mu,\nu),
\]
where the last limit follows from
\[
\left\vert W_{2}\left(  \mu_{n},\nu_{n}\right)  -W_{2}\left(  \mu,\nu\right)
\right\vert \leqslant W_{2}\left(  \mu_{n},\mu\right)  +W_{2}\left(  \nu
_{n},\nu\right)  .
\]
This proves $\left(  \ref{eq:strong-potential}\right)  $ for all
$\mathcal{P}_{2}$ targets.

\textbf{4}. We next establish that $B_{f}<\infty$. First assume without loss
of generality that $x_{\ast}=0$ in $\left(  \ref{eq:Gaussian-envelope}\right)
$ and set $C=Z^{-1}e^{-V(0)}$. For the $j$-th coordinate unit vector $e_{j},$
let $z\perp e_{j}$, the finite strongly convex function $t\mapsto V\left(
z+te_{j}\right)  $ tends to $+\infty$ as $t\rightarrow-\infty$ or
$t\rightarrow\infty,$ thus has a unique minimizer $t_{j}(z)$. Its almost
everywhere derivative changes sign only at $t_{j}\left(  z\right)  $. Hence,
for all $t,$
\[
\int_{\mathbb{R}}\left\vert \partial_{j}f\left(  z+te_{j}\right)  \right\vert
dt=\int_{-\infty}^{t_{j}\left(  z\right)  }\partial_{j}f\left(  z+te_{j}
\right)  dt-\int_{t_{j}\left(  z\right)  }^{\infty}\partial_{j}f\left(
z+te_{j}\right)  dt=2f\left(  z+t_{j}\left(  z\right)  e_{j}\right)
\leqslant2Ce^{-\kappa\left\vert z\right\vert ^{2}/2}.
\]
So Tonelli's theorem and that $\left\vert \nabla f\right\vert \leqslant
\sum_{j}\left\vert \partial_{j}f\right\vert $ yield
\begin{equation}
B_{f}=\int\left\vert \nabla f\right\vert dx\leqslant\sum_{j=1}^{d}
\int_{\mathbb{R}^{d-1}}2Ce^{-\kappa\left\vert z\right\vert ^{2}/2}
dz=2dC(2\pi/\kappa)^{(d-1)/2}<\infty. \label{eq:BV-strong}
\end{equation}
For $d=1$, the integral over $\mathbb{R}^{d-1}$ equals one. Since $f$ is
locally Lipschitz, its almost everywhere derivatives are its local weak
derivatives. Its weak gradient $\nabla f$ is also integrable by $\left(
\ref{eq:BV-strong}\right)  $ thus $f\in W^{1,1}$. So we identify $Df=\nabla
fdx$.

\textbf{5}. Finally, let both targets $\mu,\nu$ be supported in $\overline
{B}(b,L)$. Set $u=u_{\mu}-b\cdot x$, $v=u_{\nu}-b\cdot x$. The pushforward
laws by $u_{\mu},$ $u_{\nu}$ indicate that $\left\vert \nabla u\right\vert ,$
$\left\vert \nabla v\right\vert \leqslant L$, $\rho$-a.e. Since $f>0,$ the
gradient bounds also holds a.e. in $U.$ Let $w\in\left\{  u_{\mu},u_{\nu
}\right\}  $. By replacing $w\left(  x\right)  $ with
\[
w\left(  x\right)  =\sup_{z\in E_{w}}\left\{  w\left(  z\right)  +\nabla
w\left(  z\right)  \cdot\left(  x-z\right)  \right\}  ,
\]
where $E_{w}$ is a dense subset of $U$ where $w$ is differentiable, we do not
change values of $w$ in $U$ and the replacement yields a $L$-Lipschitz convex
function on $\mathbb{R}^{d}.$ So we can assume in the following that this
replacement has been performed. The mollifications of $w$ satisfy
\[
\left\vert \nabla\left(  w\ast\zeta_{\epsilon}\right)  \right\vert =\left\vert
\left(  \nabla w\right)  \ast\zeta_{\epsilon}\right\vert \leqslant\left\vert
\nabla w\right\vert \ast\zeta_{\epsilon}\leqslant L
\]
and
\[
\left\vert \left(  w\ast\zeta_{\epsilon}\right)  \left(  x\right)  -\left(
w\ast\zeta_{\epsilon}\right)  \left(  z\right)  \right\vert \leqslant
L\left\vert x-z\right\vert .
\]
As $\epsilon\rightarrow0,$ these shows that $w$ is $L$-Lipschitz. Let
\[
E_{\rho,2}(u,v)=\inf_{c\in\mathbb{R}}\Vert u-v-c\Vert_{L^{2}(\rho)}.
\]
Note that affine subtraction changes neither potential differences nor
gradient differences, $E_{\rho,2}(u,v)=E_{\rho,2}(u_{\mu},u_{\nu}).$ Also,
writing $z=u_{\mu}-u_{\nu}$ we have $\int zd\rho=0$ and for $c\in\mathbb{R}$,
\[
\int\left\vert z-c\right\vert ^{2}d\rho=\int\left\vert z\right\vert ^{2}
d\rho+c^{2}.
\]
Hence $E_{\rho,2}(u_{\mu},u_{\nu})=\Vert z\Vert_{L^{2}(\rho)}.$ Applying
Theorem~\ref{thm:interpolation} with $q=2$ therefore gives
\[
\Delta_{\rho}\left(  \mu,\nu\right)  ^{2}\leqslant8L^{4/3}B_{f}^{2/3}
E_{\rho,2}(u_{\mu},u_{\nu})^{2/3}\leqslant8L^{4/3}B_{f}^{2/3}\left(
2/\kappa\right)  ^{1/3}W_{2}\left(  \mu,\nu\right)  ^{2/3}.
\]
Taking square roots gives the coefficient $\sqrt{8}(2/\kappa)^{1/6}=2\sqrt
{2}(2/\kappa)^{1/6}$ in $\left(  \ref{eq:strong-map}\right)  $. If $W_{2}=0$,
the potentials and maps coincide by uniqueness, and the estimate still holds.
\end{proof}

\begin{remark}
The coefficient $\kappa^{-1/2}$ in Theorem \ref{thm:strong-main} is sharp.
Indeed, for the centered Gaussian with covariance $\kappa^{-1}I$ and Dirac
targets $\delta_{a},\delta_{b}$, the normalized potentials are $u_{\delta_{a}
}(x)=a\cdot x$ and $u_{\delta_{b}}(x)=b\cdot x$. Therefore
\[
\left\Vert u_{\delta_{a}}-u_{\delta_{b}}\right\Vert _{L^{2}\left(
\rho\right)  }=\kappa^{-1/2}\left\vert a-b\right\vert =\kappa^{-1/2}
W_{2}\left(  \delta_{a},\delta_{b}\right)  .
\]
So the coefficient is attained in the $\kappa$-strongly convex source class.
\end{remark}

\section{Simultaneous sharpness of the exponents}

\label{sec:Gaussian-sharp}

For source $\rho$ and targets $\mu,\nu$, we write
\[
\Delta_{\rho}\left(  \mu,\nu\right)  ^{2}=\left\Vert T_{\mu}-T_{\nu
}\right\Vert _{L^{2}\left(  \rho\right)  }.
\]

\begin{proof}
[Proof of Theorem~\ref{thm:Gaussian-main}]\textbf{1}. The construct only
focuses on the first two coordinates, the other $d-2$ coordinates integrate to
$1$ under standard Gaussian. Denote the density and distribution function of
the standard Gaussian variable by
\[
\varphi(t)=\left(  2\pi\right)  ^{-1/2}e^{-t^{2}/2},\text{ }\Phi\left(
t\right)  =\int_{-\infty}^{t}\varphi\left(  s\right)  ds,
\]
Let
\[
F\left(  t\right)  =2\int_{0}^{t}\varphi\left(  s\right)  ds.
\]
For $t\geqslant0$, we have
\begin{equation}
0\leqslant F\left(  t\right)  \leqslant ct\text{ with }c=2\varphi\left(
0\right)  >0. \label{eq:cosntant-c}
\end{equation}
and
\begin{equation}
\frac{F\left(  t\right)  }{t}=2\int_{0}^{1}\varphi(ts)ds\rightarrow c\text{ as
}t\rightarrow0. \label{eq:Gaussian-F-limit}
\end{equation}
Let $Z$ be a one dimensional standard Gaussian variable. For fixed $h>0$, the
function
\begin{equation}
G\left(  a\right)  =\mathbb{E}F\left(  \left(  a+hZ\right)  _{+}\right)
\label{eq:G-definition}
\end{equation}
is continuous with limits $0,1$ at $-\infty,\infty$ respectively by the
dominated convergence theorem. If $a<b$, then $0<a+hZ<b+hZ$ on $\left\{
Z>-a/h\right\}  ,$ which has positive probability. Hence the strict inequality
holds
\[
G\left(  a\right)  <G\left(  b\right)  .
\]
Thus $G\left(  a\right)  $ is strictly increasing in $a$, and there is a
unique $a_{h}\in\mathbb{R}$ satisfying
\begin{equation}
G\left(  a_{h}\right)  =F\left(  h\right)  \triangleq m_{h}\in\left(
0,1\right)  . \label{eq:Gaussian-mass-equation}
\end{equation}

\textbf{2}. We now give a rough estimate about where $a_{h}$ locates. Define
\[
H\left(  A\right)  =\mathbb{E}\left[  \left(  A+Z\right)  _{+}\right]  ,\text{
}\forall A\in\mathbb{R}.
\]
Since $\varphi^{\prime}\left(  z\right)  =-z\varphi\left(  z\right)  ,$ we
have
\begin{equation}
H\left(  A\right)  =\int_{-A}^{\infty}\left(  A+z\right)  \varphi\left(
z\right)  dz=A\Phi\left(  A\right)  +\varphi\left(  A\right)  ,\nonumber
\end{equation}
hence
\[
H^{\prime}\left(  A\right)  =\Phi\left(  A\right)  >0.
\]
Moreover
\[
H\left(  0\right)  =\varphi\left(  0\right)  <1\text{ and }H\left(  1\right)
=\mathbb{E}\left(  1+Z\right)  +\mathbb{E}\left[  \left(  -1-Z\right)
_{+}\right]  >1.
\]
Therefore there is a unique $A_{\ast}\in\left(  0,1\right)  $ with
\begin{equation}
H\left(  A_{\ast}\right)  =\mathbb{E}\left(  A_{\ast}+Z\right)  _{+}=1.
\label{eq:Gaussian-A}
\end{equation}
For fixed $A$, $\lim_{h\rightarrow0}F\left(  h\left(  A+Z\right)  _{+}\right)
/h$ is bounded by $c(A+Z)_{+}$. The dominated convergence theorem yields
\begin{equation}
\lim_{h\rightarrow0}\frac{\mathbb{E}F\left(  h\left(  A+Z\right)  _{+}\right)
}{h}=cH\left(  A\right)  . \label{eq:Gaussian-scaled-mass}
\end{equation}
Evaluating the limit at $A=0$ and $A=1$ yield, as $h\rightarrow0,$
\[
\frac{G\left(  0\right)  }{h}=\frac{\mathbb{E}F\left(  hZ_{+}\right)  }
{h}\rightarrow cH\left(  0\right)  <c\text{ and }\frac{G\left(  h\right)  }
{h}=\frac{\mathbb{E}F\left(  h\left(  1+Z\right)  _{+}\right)  }{h}\rightarrow
cH\left(  1\right)  >c.
\]
While $\left(  \ref{eq:Gaussian-F-limit}\right)  $ and $\left(
\ref{eq:Gaussian-mass-equation}\right)  $ imply, as $h\rightarrow0,$
\[
\frac{G\left(  a_{h}\right)  }{h}=\frac{F\left(  h\right)  }{h}\rightarrow
c=cH\left(  A_{\ast}\right)  .
\]
Hence
\begin{equation}
\frac{m_{h}}{h}\rightarrow c\text{ as }h\rightarrow0\text{,}
\label{eq:Gaussian-mh-limit}
\end{equation}
and for small $h>0,$
\[
\frac{G\left(  0\right)  }{h}<\frac{m_{h}}{h}<\frac{G\left(  h\right)  }{h}.
\]
This indicates that $0<a_{h}<h$ for small $h>0$. Moreover, for $0<\epsilon
<\min\left(  A_{\ast},1-A_{\ast}\right)  $, since $H$ increases strictly, we
have
\[
cH\left(  A_{\ast}-\epsilon\right)  <c=cH\left(  A_{\ast}\right)  <cH\left(
A_{\ast}+\epsilon\right)  .
\]
Since $F\left(  h\right)  =G\left(  a_{h}\right)  =\mathbb{E}F(h\left(
a_{h}/h+Z\right)  _{+}),$ we get from $\left(  \ref{eq:Gaussian-F-limit}
\right)  $, $\left(  \ref{eq:Gaussian-scaled-mass}\right)  $ that, for small
$h>0$,
\[
\mathbb{E}F\left(  h\left(  A_{\ast}-\epsilon+Z\right)  _{+}\right)  <F\left(
h\right)  =\mathbb{E}F(h\left(  a_{h}/h+Z\right)  _{+})<\mathbb{E}F\left(
h\left(  A_{\ast}+\epsilon+Z\right)  _{+}\right)  .
\]
The same argument that was used to show the strict increase of $G\left(
\cdot\right)  $ also shows that the function $x\mapsto\mathbb{E}F(h\left(
x+Z\right)  _{+})$ increases strictly$,$ so $A_{\ast}-\epsilon<a_{h}
/h<A_{\ast}+\epsilon$. Consequently
\begin{equation}
\frac{a_{h}}{h}\rightarrow A_{\ast}\text{ as }h\rightarrow0.
\label{eq:Gaussian-A-limit}
\end{equation}

\textbf{3}. Define the finite convex functions
\begin{equation}
u_{h}\left(  x\right)  =\max\left\{  -x_{1},x_{1},h\right\}  ,\text{ }
v_{h}\left(  x\right)  =\max\left\{  -x_{1},x_{1},a_{h}+hx_{2}\right\}  .
\label{eq:Gaussian-planes}
\end{equation}
Let
\[
A_{h}=\left\{  \left\vert x_{1}\right\vert <h\right\}  ,\text{ }B_{h}=\left\{
\left\vert x_{1}\right\vert <\left(  a_{h}+hx_{2}\right)  _{+}\right\}  .
\]
By $\left(  \ref{eq:Gaussian-mass-equation}\right)  $ and the independence of
the Gaussian coordinates, we have
\[
\gamma_{d}\left(  A_{h}\right)  =F\left(  h\right)  =m_{h},\text{ }\gamma
_{d}\left(  B_{h}\right)  =\int F\left(  \left(  a_{h}+hz\right)  _{+}\right)
\varphi\left(  z\right)  \,dz=m_{h}.
\]
Let $e_{1},$ $e_{2}$ denote the unit vectors of the first two coordinates.
Thus the gradient push-forwards are
\[
\mu_{h}\triangleq\left(  \nabla u_{h}\right)  _{\#}\gamma_{d}=\frac{1-m_{h}
}{2}\left(  \delta_{-e_{1}}+\delta_{e_{1}}\right)  +m_{h}\delta_{0},
\]
\[
\nu_{h}\triangleq\left(  \nabla v_{h}\right)  _{\#}\gamma_{d}=\frac{1-m_{h}
}{2}\left(  \delta_{-e_{1}}+\delta_{e_{1}}\right)  +m_{h}\delta_{he_{2}}.
\]
For $0<h<1$ all supports lie in the closed unit ball. An explicit coupling can
be constructed
\[
\pi_{h}=\frac{1-m_{h}}{2}\left(  \delta_{\left(  -e_{1},-e_{1}\right)
}+\delta_{\left(  e_{1},e_{1}\right)  }\right)  +m_{h}\delta_{\left(
0,he_{2}\right)  }\in\Pi\left(  \mu_{h},\nu_{h}\right)  .
\]
Writing $y,$ $z$ for the marginal variables$.$ Clearly
\[
\int\left\vert y-z\right\vert ^{2}d\pi_{h}\left(  y,z\right)  =h^{2}m_{h}.
\]
On the other hand, for any $\pi\in\Pi\left(  \mu_{h},\nu_{h}\right)  $, its
first marginal has $y_{2}=0$, thus
\[
\int\left\vert y-z\right\vert ^{2}d\pi\left(  y,z\right)  \geqslant
\int\left\vert y_{2}-z_{2}\right\vert ^{2}d\pi\left(  y,z\right)
=\int\left\vert z_{2}\right\vert ^{2}d\nu_{h}\left(  z\right)  =h^{2}m_{h}.
\]
Hence the explicit coupling has the minimal cost,
\begin{equation}
W_{2}\left(  \mu_{h},\nu_{h}\right)  =hm_{h}^{1/2}. \label{eq:Gaussian-Wr}
\end{equation}
The optimal maps from $\gamma_{d}$ to $\mu_{h},\nu_{h}$ can be written out
explicitly
\[
T_{\mu_{h}}\left(  x\right)  =\operatorname{sgn}\left(  x_{1}\right)
e_{1}1_{A_{h}^{c}}\left(  x\right)  ,\text{ }T_{\nu_{h}}\left(  x\right)
=\operatorname{sgn}\left(  x_{1}\right)  e_{1}1_{B_{h}^{c}}\left(  x\right)
+he_{2}1_{B_{h}}\left(  x\right)  .
\]
The difference is
\[
T_{\mu_{h}}-T_{\nu_{h}}=\operatorname{sgn}\left(  x_{1}\right)  e_{1}\left(
1_{B_{h}}-1_{A_{h}}\right)  -he_{2}1_{B_{h}},
\]
Using that $e_{1}\cdot e_{2}=0,$ we have
\[
\left\vert T_{\mu_{h}}-T_{\nu_{h}}\right\vert ^{2}=1_{A_{h}\Delta B_{h}}
+h^{2}1_{B_{h}}.
\]
Thus using Tonelli's theorem,
\[
\int\left\vert T_{\mu_{h}}-T_{\nu_{h}}\right\vert ^{2}d\rho=\int_{\mathbb{R}
}\left[  \int_{\mathbb{R}}1_{A_{h}\Delta B_{h}}d\gamma_{1}\left(
x_{1}\right)  +h^{2}1_{B_{h}}\right]  d\gamma_{2}\left(  x_{2}\right)  .
\]
For the evaluation of double integral, fix $x_{2}=z$, the integration of
$x_{1}$ gives the mass $\left\vert F\left(  h\right)  -F\left(  \left(
a_{h}+hz\right)  _{+}\right)  \right\vert $ to $A_{h}\Delta B_{h}.$ Then
integrating $x_{2}$ yields
\begin{equation}
\Delta_{\gamma_{d}}\left(  \mu_{h},\nu_{h}\right)  ^{2}=\mathbb{E}\left\vert
F\left(  h\right)  -F\left(  \left(  a_{h}+hz\right)  _{+}\right)  \right\vert
+h^{2}m_{h}. \label{eq:Gaussian-map-exact}
\end{equation}

\textbf{4}. We now compute the order of $\Delta_{\gamma_{d}}\left(  \mu
_{h},\nu_{h}\right)  ^{2}$ as $h\rightarrow0.$ Let $Q=\left(  A_{\ast
}+Z\right)  _{+}$. By $\left(  \ref{eq:Gaussian-A-limit}\right)  ,$ for
sufficiently small $h$, $0<a_{h}/h<1$, and
\[
0\leqslant\frac{\left\vert F\left(  h\right)  -F\left(  h\left(
a_{h}/h+hz\right)  _{+}\right)  \right\vert }{h}\leqslant c\left[  1+\left(
a_{h}/h+Z\right)  _{+}\right]  \leqslant c\left(  2+\left\vert Z\right\vert
\right)  .
\]
By $\left(  \ref{eq:Gaussian-F-limit}\right)  ,$ we have, if $\left(
a_{h}/h+z\right)  _{+}>0,$
\[
\frac{F\left(  h\right)  -F\left(  \left(  a_{h}+hz\right)  _{+}\right)  }
{h}=\frac{F\left(  h\right)  }{h}-\frac{F\left(  h\left(  a_{h}/h+z\right)
_{+}\right)  }{h\left(  a_{h}/h+z\right)  _{+}}\left(  a_{h}/h+z\right)
_{+}\rightarrow c\left(  1-\left(  A_{\ast}+z\right)  _{+}\right)  .
\]
The limit clearly still holds even if $\left(  a_{h}/h+z\right)  _{+}=0$. Also
$h^{2}m_{h}/h=hm_{h}\rightarrow0$. Hence dividing $\left(
\ref{eq:Gaussian-map-exact}\right)  $ by $h$ and taking limit $h\rightarrow0,$
we obtain that
\begin{equation}
\frac{\Delta_{\gamma_{d}}\left(  \mu_{h},\nu_{h}\right)  ^{2}}{h}\rightarrow
c\mathbb{E}\left\vert 1-Q\right\vert \triangleq c_{m}\text{.}
\label{eq:Gaussian-map-limit}
\end{equation}
The constant $c_{m}$ is strictly positive$.$ Indeed, $c>0$ by $\left(
\ref{eq:cosntant-c}\right)  ,$ the event $\left\{  Z<-A_{\ast}\right\}  $ has
strictly positive probability and on it $Q=0$. So
\[
c_{m}=c\mathbb{E}\left\vert 1-Q\right\vert \geqslant c\mathbb{E}\left[
\left\vert 1-Q\right\vert 1_{\left\{  Z<-A_{\ast}\right\}  }\right]  =c>0.
\]

\textbf{5}. We next compute the order of $\left\Vert u_{\mu_{h}}-u_{\nu_{h}
}\right\Vert _{L^{2}\left(  \gamma_{d}\right)  }^{2}$ and its normalized
version w.r.t. $h\rightarrow0.$ Define
\[
D_{A}\left(  s,z\right)  =\max\left\{  \left\vert s\right\vert ,1\right\}
-\max\left\{  \left\vert s\right\vert ,A+z\right\}  .
\]
Then $\left(  u_{h}-v_{h}\right)  \left(  hs,z\right)  =hD_{a_{h}/h}\left(
s,z\right)  $. Substituting $x_{1}=hs$ yields
\begin{equation}
\int\left\vert u_{h}-v_{h}\right\vert ^{2}d\gamma_{d}=h^{3}\int_{\mathbb{R}
^{2}}D_{a_{h}/h}\left(  s,z\right)  ^{2}\varphi\left(  hs\right)
\varphi\left(  z\right)  dsdz. \label{eq:Gaussian-potential-exact}
\end{equation}
For $0\leqslant A\leqslant1$, $\left\vert D_{A}(s,z)\right\vert \leqslant
2+\left\vert z\right\vert $. Also $D_{A}(s,z)=0$ if $\left\vert s\right\vert
>1+\left\vert z\right\vert $. Thus the integrand of $\left(
\ref{eq:Gaussian-potential-exact}\right)  $\ has
\[
D_{a_{h}/h}\left(  s,z\right)  ^{2}\varphi\left(  hs\right)  \varphi\left(
z\right)  \leqslant\varphi\left(  0\right)  \left(  2+\left\vert z\right\vert
\right)  ^{2}1_{\left\{  \left\vert s\right\vert \leqslant1+\left\vert
z\right\vert \right\}  }\varphi\left(  z\right)  .
\]
The upper bound is $dsdz$-integrable with integral at most $2\varphi\left(
0\right)  \mathbb{E}\left[  \left(  1+\left\vert Z\right\vert \right)  \left(
2+\left\vert Z\right\vert \right)  ^{2}\right]  <\infty$. By $\left(
\ref{eq:Gaussian-A-limit}\right)  $ and the dominated convergence theorem, as
$h\rightarrow0,$
\begin{equation}
\frac{\int\left\vert u_{h}-v_{h}\right\vert ^{2}d\gamma_{d}}{h^{3}}
\rightarrow\varphi\left(  0\right)  \int_{\mathbb{R}^{2}}D_{A_{\ast}}\left(
s,z\right)  ^{2}\varphi\left(  z\right)  dsdz\triangleq c_{p}
\label{eq:Gaussian-potential-limit}
\end{equation}
We show that $c_{p}>0.$ Since $\varphi\left(  0\right)  >0,$ it suffices to
look at the integral
\[
\int_{\mathbb{R}^{2}}D_{A_{\ast}}\left(  s,z\right)  ^{2}\varphi\left(
z\right)  dsdz=\int_{\mathbb{R}}\left[  \int_{\mathbb{R}}D_{A_{\ast}}\left(
s,z\right)  ^{2}ds\right]  \varphi\left(  z\right)  dz.
\]
If $A_{\ast}+z\leqslant1,$ $s\geqslant0$, we explicitly have
\[
D_{A_{\ast}}\left(  s,z\right)  =\max\left\{  s,1\right\}  -\max\left\{
s,A_{\ast}+z\right\}  =\left\{
\begin{array}
[c]{ll}
1-\left(  A_{\ast}+z\right)  , & 0\leqslant s\leqslant A_{\ast}+z,\\
1-s, & A_{\ast}+z<s\leqslant1,\\
0, & s>1.
\end{array}
\right.  .
\]
Hence
\begin{align*}
\int_{\mathbb{R}^{2}}D_{A_{\ast}}\left(  s,z\right)  ^{2}\varphi\left(
z\right)  dsdz  &  \geqslant\int\left[  \int_{0}^{\infty}D_{A_{\ast}}\left(
s,z\right)  ^{2}ds\right]  1_{\left\{  z\leqslant1-A_{\ast}\right\}  }
\varphi\left(  z\right)  dz\\
&  \geqslant\int\left[  \int_{A_{\ast}+z}^{1}\left(  1-s\right)
^{2}ds\right]  1_{\left\{  z\leqslant-A_{\ast}\right\}  }\varphi\left(
z\right)  dz>0.
\end{align*}
Thus $c_{p}>0.$ Following the same computation as the squared case
$\int\left\vert u_{h}-v_{h}\right\vert ^{2}d\gamma_{d}$, we obtain
\[
\left\vert \int\left(  u_{h}-v_{h}\right)  d\gamma_{d}\right\vert
=h^{2}\left\vert \int G_{a_{h}/h}\left(  s,z\right)  \varphi\left(  hs\right)
\varphi\left(  z\right)  dsdz\right\vert \leqslant2\varphi\left(  0\right)
h^{2}\mathbb{E}\left[  \left(  1+\left\vert Z\right\vert \right)  \left(
2+\left\vert Z\right\vert \right)  \right]  =O\left(  h^{2}\right)  .
\]
Therefore
\begin{align*}
\left\Vert \left(  u_{\mu_{h}}-u_{\nu_{h}}\right)  -\int\left(  u_{h}
-v_{h}\right)  d\gamma_{d}\right\Vert _{L^{2}\left(  \gamma_{d}\right)  }^{2}
&  =\int\left\vert u_{h}-v_{h}\right\vert ^{2}d\gamma_{d}-\left(  \int\left(
u_{h}-v_{h}\right)  d\gamma_{d}\right)  ^{2}\\
&  =h^{3}c_{p}+o\left(  h^{3}\right)  -O\left(  h^{4}\right)  =h^{3}
c_{p}+o\left(  h^{3}\right)  .
\end{align*}
This indicates that normalizing $u_{h},$ $v_{h}$ to have zero means has the
same order as $\left\Vert u_{\mu_{h}}-u_{\nu_{h}}\right\Vert _{L^{2}\left(
\gamma_{d}\right)  }^{2}$ w.r.t. $h\rightarrow0.$

\textbf{6}. Finally we collect all asymptotic quantities computed so far.
Specifically, we have
\[
\begin{array}
[c]{rll}
\left(  \ref{eq:Gaussian-mh-limit}\right)  &  & m_{h}/h\rightarrow c\\
\left(  \ref{eq:Gaussian-A-limit}\right)  &  & a_{h}/h\rightarrow A_{\ast}\\
\left(  \ref{eq:Gaussian-Wr}\right)  &  & W_{2}\left(  \mu_{h},\nu_{h}\right)
=hm_{h}^{1/2}\\
\left(  \ref{eq:Gaussian-map-limit}\right)  &  & \Delta_{\gamma_{d}}\left(
\mu_{h},\nu_{h}\right)  ^{2}=hc_{m}+o\left(  h\right) \\
\left(  \ref{eq:Gaussian-potential-limit}\right)  &  & \left\Vert u_{\mu_{h}
}-u_{\nu_{h}}\right\Vert _{L^{2}\left(  \gamma_{d}\right)  }^{2}=h^{3}
c_{p}+o\left(  h^{3}\right)
\end{array}
\]
The last row remains identical for potential that are normalized to have zero
means. All constants $c,A_{\ast},c_{m},c_{p}$ are strictly positive. Therefore
the desired asymptotics are proved with with $\kappa_{0}=\sqrt{c}$,
$\kappa_{1}=\sqrt{c_{m}}$, $\kappa_{2}=\sqrt{c_{p}}$. The constructed family
also gives linear potential stability
\[
\frac{\left\Vert u_{\mu_{h}}-u_{\nu_{h}}\right\Vert _{L^{2}\left(  \gamma
_{d}\right)  }}{W_{2}\left(  \mu_{h},\nu_{h}\right)  }=\frac{h^{3/2}
c_{p}^{1/2}+o\left(  h^{3/2}\right)  }{h^{3/2}\left(  m_{h}/h\right)  ^{1/2}
}\sim\frac{c_{p}^{1/2}}{c^{1/2}}\text{ as }h\rightarrow0
\]
and realizes the $1/3$ power conversion from potential stability to map
stability,
\[
\frac{\Delta_{\gamma_{d}}\left(  \mu_{h},\nu_{h}\right)  }{\left\Vert
u_{\mu_{h}}-u_{\nu_{h}}\right\Vert _{L^{2}\left(  \gamma_{d}\right)  }^{1/3}
}=\frac{h^{1/2}c_{m}^{1/2}+o\left(  h^{1/2}\right)  }{h^{1/2}c_{p}
^{1/6}+o\left(  h^{1/2}\right)  }\sim\frac{c_{m}^{1/2}}{c_{p}^{1/6}}\text{ as
}h\rightarrow0.
\]
Moreover, for any $\alpha>1/3$,
\[
\frac{\Delta_{\gamma_{d}}\left(  \mu_{h},\nu_{h}\right)  }{W_{2}\left(
\mu_{h},\nu_{h}\right)  ^{\alpha}}=\frac{\left(  hc_{m}+o\left(  h\right)
\right)  ^{1/2}}{h^{\alpha}m_{h}^{\alpha/2}}=\frac{h^{1/2}c_{m}^{1/2}+o\left(
h^{1/2}\right)  }{h^{3\alpha/2}\left(  m_{h}/h\right)  ^{\alpha/2}}\sim
\frac{c_{m}^{1/2}}{c^{\alpha/2}}h^{\left(  1-3\alpha\right)  /2}
\rightarrow+\infty\text{ as }h\rightarrow0.
\]

\end{proof}

\subsection{Gaussian dimension dependence}

For $d\geqslant2$, let $C_{d}$ denote the least constant in
\[
\left\Vert T_{\mu}-T_{\nu}\right\Vert _{L^{2}(\gamma_{d})}\leqslant C_{d}
W_{2}\left(  \mu,\nu\right)  ^{1/3}\text{ for all }\mu,\nu\in\mathcal{P}
\left(  \overline{B}\left(  0,1\right)  \right)  .
\]
Here $\mathcal{P}\left(  \overline{B}\left(  0,1\right)  \right)  $ denotes
the set of probabilities supported in $\overline{B}\left(  0,1\right)  $. The
estimate $\left(  \ref{eq:Gaussian-upper}\right)  $ implies
\[
C_{d}\leqslant2^{5/3}d^{1/6}.
\]
On the other hand, \cite[Proposition 5]{wang2026quantitative} gives an
dimension lower bound: there are universal constants $c_{\ast},C_{\ast}>0$
such that, for all sufficiently large $d$, there are targets $\mu_{d},\nu_{d}$
in the unit ball satisfying
\[
\left\Vert T_{\mu_{d}}-T_{\nu_{d}}\right\Vert _{L^{2}(\gamma_{d})}\geqslant
c_{\ast}\text{ and }W_{2}\left(  \mu_{d},\nu_{d}\right)  \leqslant C_{\ast
}d^{-1/2}.
\]
These targets satisfy exactly the source and support conditions in the
definition of $C_{d}$. Their distance $W_{2}\left(  \mu_{d},\nu_{d}\right)  $
is positive by uniqueness of the Brenier map. Substitution yields
\[
c_{\ast}\leqslant C_{d}(C_{\ast}d^{-1/2})^{1/3}.
\]
It follows that
\[
c_{\ast}C_{\ast}^{-1/3}d^{1/6}\leqslant C_{d}\leqslant2^{5/3}d^{1/6}.
\]
Thus the power $d^{1/6}$ is optimal for Theorem \ref{thm:strong-main}.

For the standard Gaussian density , the constant can be derived explicitly.
The density and its gradient are
\[
f_{d}(x)=\left(  2\pi\right)  ^{-d/2}e^{-\left\vert x\right\vert ^{2}
/2},\text{ }\nabla f_{d}\left(  x\right)  =-xf_{d}\left(  x\right)  .
\]
Its source potential $V(x)=\left\vert x\right\vert ^{2}/2$ satisfies $\left(
\ref{eq:strong-source}\right)  $ with $\kappa=1$. For the one-dimensional
density $\varphi\left(  t\right)  =\left(  2\pi\right)  ^{-1/2}e^{-t^{2}/2}$,
integration by parts gives
\[
\int_{\mathbb{R}}t^{2}\varphi(t)dt=-\int_{\mathbb{R}}t\varphi^{\prime
}(t)dt=-[t\varphi\left(  t\right)  ]_{-\infty}^{\infty}+\int_{\mathbb{R}
}\varphi\left(  t\right)  dt=1.
\]
Consequently
\begin{align*}
B_{f_{d}}  &  =\int_{\mathbb{R}^{d}}\left\vert \nabla f_{d}\left(  x\right)
\right\vert dx=\int_{\mathbb{R}^{d}}\left\vert x\right\vert d\gamma
_{d}(x)\leqslant\left(  \int_{\mathbb{R}^{d}}\left\vert x\right\vert
^{2}d\gamma_{d}(x)\right)  ^{1/2}\left(  \int_{\mathbb{R}^{d}}d\gamma
_{d}\left(  x\right)  \right)  ^{1/2}\\
&  =\left(  \sum_{j=1}^{d}\int_{\mathbb{R}^{d}}x_{j}^{2}d\gamma_{d}(x)\right)
^{1/2}=\sqrt{d}.
\end{align*}
Since $f_{d}$ is smooth and its first derivatives are integrable, this also
proves $f_{d}\in W^{1,1}(\mathbb{R}^{d})$. If both targets are supported in
$\overline{B}(b,L)$, with $b\in\mathbb{R}^{d}$ and $L>0$, substitution in
$\left(  \ref{eq:strong-map}\right)  $ yields
\[
\left\Vert T_{\mu}-T_{\nu}\right\Vert _{L^{2}(\gamma_{d})}\leqslant2\sqrt
{2}2^{1/6}L^{2/3}B_{f_{d}}^{1/3}W_{2}\left(  \mu,\nu\right)  ^{1/3}
\leqslant2^{\,1+1/2+1/6}(\sqrt{d})^{1/3}L^{2/3}W_{2}\left(  \mu,\nu\right)
^{1/3}.
\]
Since $1+1/2+1/6=5/3$ and $(\sqrt{d})^{1/3}=d^{1/6}$, this is
\begin{equation}
\Delta_{\gamma_{d}}\left(  \mu,\nu\right)  \leqslant2^{5/3}d^{1/6}L^{2/3}
W_{2}\left(  \mu,\nu\right)  ^{1/3}. \label{eq:Gaussian-upper}
\end{equation}
In dimension one the monotone coupling gives $\Delta_{\rho}=W_{2}$ (see
\cite[Theorem 2.17]{santambrogio2015optimal}\cite[Remark 2.19]
{villani2003topics}), so the sharpness claim is restricted to $d\geqslant2$.

\appendix

\section{Potential interpolation with source of bounded variation}

\label{sec:BV}

This section gives the conversion from potential stability to Brenier map stability.

\begin{lemma}
\label{lem:BV}Let $f\in BV(\mathbb{R}^{d})$ be nonnegative$,$ $L\geqslant0$
and $u$ be $L$-Lipschitz convex on $\mathbb{R}^{d}$. Let $B_{f}=\left\vert
Df\right\vert \left(  \mathbb{R}^{d}\right)  $ and $e$ be a unit vector.
Whenever $u$ is differentiable at $x,$ we define for $t>0$,
\[
R_{u}\left(  x,te\right)  =u\left(  x+te\right)  -u\left(  x\right)  -t\nabla
u\left(  x\right)  \cdot e.
\]
Then $R_{u}\left(  x,te\right)  \geqslant0$ and
\begin{equation}
\int R_{u}\left(  x,te\right)  f\left(  x\right)  dx\leqslant\frac{1}{2}
LB_{f}t^{2}. \label{eq:remainder}
\end{equation}
If $u$ is smooth, then for $z\in\mathbb{R}^{d}$,
\begin{equation}
\int f(x)\Delta u\left(  x+z\right)  dx=-\int\nabla u\left(  x+z\right)  \cdot
dDf(x)\leqslant LB_{f}. \label{eq:laplacian}
\end{equation}

\end{lemma}

\begin{proof}
\textbf{1}. First assume that $u$ is smooth. Let $\chi_{R}\in C_{c}^{\infty
}(\mathbb{R}^{d})$ equal one on $B\left(  0,R\right)  $, vanish outside
$B\left(  0,2R\right)  $, and satisfy $0\leqslant\chi_{R}\leqslant1$,
$\left\vert \nabla\chi_{R}\right\vert \leqslant C/R$ for some constant
independent of $R$. Such function exists, e.g. we may take the smooth function
$H$ from \cite[Lemma 2.22]{lee2013smooth} and let $\chi_{R}\left(  x\right)
=H\left(  \left\vert x\right\vert /R\right)  .$ For fixed $z,$ by integration
by parts (see \cite[Section 3.1]{ambrosio2000functions}) we have
\begin{equation}
\int\chi_{R}f\Delta u\left(  \left(  x+z\right)  \right)  dx=-\int\chi
_{R}\nabla u\left(  x+z\right)  \cdot dDf\left(  x\right)  -\int f\nabla
\chi_{R}\cdot\nabla u\left(  \,x+z\right)  dx. \label{eq:integration-by-parts}
\end{equation}
The second term on the right side converges to zero, because
\begin{equation}
\left\vert \int f\nabla\chi_{R}\cdot\nabla u\left(  x+z\right)  dx\right\vert
\leqslant\frac{CL}{R}\left\Vert f\right\Vert _{1}. \label{eq:cutoff}
\end{equation}
The integrand in the first term converges pointwise to $\nabla u\left(
x+z\right)  $ and is bounded by $L,$ so the dominated convergence theorem
yields
\[
\lim_{R\rightarrow\infty}\int\chi_{R}\nabla u\left(  x+z\right)  \cdot
dDf\left(  x\right)  =\int\nabla u\left(  x+z\right)  \cdot dDf\left(
x\right)  .
\]
Since $u$ is smooth convex, $\Delta u\geqslant0$. Thus using Fatou's lemma$,$
we obtain
\[
0\leqslant\int f\Delta u\left(  \left(  x+z\right)  \right)  dx\leqslant
\lim_{R\rightarrow\infty}\int\chi_{R}f\Delta u\left(  \left(  x+z\right)
\right)  dx=-\int\nabla u\left(  x+z\right)  \cdot dDf\left(  x\right)
\leqslant LB_{f}.
\]
This also shows that $f\Delta u\left(  \left(  x+z\right)  \right)  $ is
integrable. Hence, using the dominated convergence theorem in $\left(
\ref{eq:integration-by-parts}\right)  $ again proves $\left(
\ref{eq:laplacian}\right)  $. \ For the directional derivative $\partial
_{ee}u$, $0\leqslant\partial_{ee}u\leqslant\Delta u$. Using Taylor's formula
with integral remainer and Tonelli's theorem, we have
\[
\int R_{u}\left(  x,te\right)  f\left(  x\right)  dx=\int_{0}^{t}\left(
t-s\right)  \int f(x)\partial_{ee}u\left(  x+se\right)  dxds\leqslant
LB_{f}\int_{0}^{t}\left(  t-s\right)  ds=\frac{1}{2}LB_{f}t^{2}.
\]
The inequality uses $\left(  \ref{eq:laplacian}\right)  $.

\textbf{2}. \ For general convex $u$, fix a nonnegative $\zeta\in
C_{c}^{\infty}\left(  B\left(  0,1\right)  \right)  $ with $\int\zeta=1$,
$\epsilon>0,$ $\zeta_{\epsilon}(z)=\epsilon^{-d}\zeta\left(  z/\epsilon
\right)  $ and let $u_{\epsilon}=u\ast\zeta_{\epsilon}$. Thus $\zeta
_{\epsilon}$ is a smooth mollifier supported in $B(0,\epsilon)$. All
convolutions are finite because $u$ is locally bounded. For $0\leqslant
\lambda\leqslant1$,
\[
u_{\epsilon}\left(  \lambda x+\left(  1-\lambda\right)  y\right)  =\int
u\left(  \lambda\left(  x-z\right)  +\left(  1-\lambda\right)  \left(
y-z\right)  \right)  \zeta_{\epsilon}\left(  z\right)  dz\leqslant\lambda
u_{\epsilon}\left(  x\right)  +(1-\lambda)u_{\epsilon}\left(  y\right)  .
\]
So the smoothed function $u_{\epsilon}$ is still convex. Moreover
\[
\left\vert u_{\epsilon}\left(  x\right)  -u_{\epsilon}\left(  y\right)
\right\vert \leqslant\int L\left\vert x-y\right\vert \zeta_{\epsilon}\left(
z\right)  dz=L\left\vert x-y\right\vert ,
\]
\[
\left\vert u_{\epsilon}\left(  x\right)  -u\left(  x\right)  \right\vert
\leqslant\int L\left\vert z\right\vert \zeta_{\epsilon}\left(  z\right)
dz\leqslant L\epsilon.
\]
Thus $u_{\epsilon}$ is $L$-Lipschitz with a uniform approximation error. For
almost all $x,$ $\nabla u_{\epsilon}=\left(  \nabla u\right)  \ast
\zeta_{\epsilon}$ exists and $\nabla u_{\epsilon}\left(  x\right)
\rightarrow\nabla u\left(  x\right)  $ (see \cite[C.5 Theorem 7 (iv)]
{evans2022partial}). Therefore for each such $x$, fixed $t>0$ and any unit
vector $e$, we have
\begin{align*}
\left\vert R_{u_{\epsilon}}\left(  x,te\right)  -R_{u}(x,te)\right\vert  &
\leqslant\left\vert u_{\epsilon}(x+te)-u(x+te)\right\vert +\left\vert
u_{\epsilon}\left(  x\right)  -u\left(  x\right)  \right\vert +t\left\vert
\nabla u_{\epsilon}\left(  x\right)  -\nabla u\left(  x\right)  \right\vert \\
&  \leqslant2L\epsilon+t\left\vert \nabla u_{\epsilon}\left(  x\right)
-\nabla u\left(  x\right)  \right\vert \rightarrow0.
\end{align*}
The potential approximation is uniform at $x+te$. Using convexity and the
Lipschitz bound we get
\[
0\leqslant R_{u_{\epsilon}}\left(  x,te\right)  \leqslant\left\vert
u_{\epsilon}\left(  x+te\right)  -u_{\epsilon}\left(  x\right)  \right\vert
+t\left\vert \nabla u_{\epsilon}\left(  x\right)  \right\vert \leqslant2Lt.
\]
Since $f\in L^{1}$, the dominated convergence theorem yields
\[
\int R_{u}\left(  x,te\right)  f\left(  x\right)  dx=\lim_{\epsilon
\rightarrow0}\int R_{u_{\epsilon}}\left(  x,te\right)  f\left(  x\right)
dx\leqslant\tfrac{1}{2}LB_{f}t^{2}.
\]
Therefore $\left(  \ref{eq:remainder}\right)  $ holds in the nonsmooth case.
\end{proof}

\begin{theorem}
[Potential to map interpolation]\label{thm:interpolation} Let $f\in
BV(\mathbb{R}^{d})$ be nonnegative, $\rho=fdx,$ $L\geqslant0$ and $u,$
$v:\mathbb{R}^{d}\mapsto\mathbb{R}$ be $L$-Lipschitz convex. Let
$B_{f}=\left\vert Df\right\vert \left(  \mathbb{R}^{d}\right)  $ and
\[
E_{f}=\int\left\vert \nabla u-\nabla v\right\vert ^{2}fdx.
\]
For $1\leqslant q<\infty$, it holds that
\begin{equation}
E_{f}\leqslant8L^{\left(  q+2\right)  /\left(  q+1\right)  }B_{f}^{q/\left(
q+1\right)  }\inf_{c\in\mathbb{R}}\left(  \int\left\vert u-v-c\right\vert
^{q}fdx\right)  ^{1/\left(  q+1\right)  }. \label{eq:interpolation}
\end{equation}
For $q=\infty,$ the $q$-norm becomes essential supremum with respect to $\rho
$,
\begin{equation}
E_{f}\leqslant4LB_{f}\inf_{c\in\mathbb{R}}\left\Vert u-v-c\right\Vert
_{L^{\infty}(\rho)} \label{eq:infinity}
\end{equation}
whenever the infimum is finite. If $f=0$ a.e. or $L=0$, then $E_{f}=0$.
\end{theorem}

\begin{proof}
\textbf{1}. If $L=0$ or $f=0$ almost everywhere, the gradients agree or the
weight vanishes, respectively, and $E_{f}=0$. Suppose henceforth that $L>0$
and $f$ is nonzero.

\textbf{2.} First assume that $u,v$ are smooth. Fix $c\in\mathbb{R}$ and write
$w=u-v-c$. For $t>0,$ define
\[
\tau_{t}\left(  s\right)  =\max\left\{  -t,\min\left\{  s,t\right\}  \right\}
.
\]
We prove by smoothing that
\begin{align}
\int_{\left\{  \left\vert w\right\vert <t\right\}  }\left\vert \nabla
w\right\vert ^{2}fdx  &  =\int f\nabla\left(  \tau_{t}\left(  w\right)
\right)  \cdot\nabla wdx\label{eq:low-level}\\
&  =-\int f\tau_{t}\left(  w\right)  \Delta wdx-\int\tau_{t}\left(  w\right)
\nabla w\cdot dDf\nonumber\\
&  \leqslant t\int f\left(  \Delta u+\Delta v\right)  dx+2LtB_{f}
\leqslant4LB_{f}t.\nonumber
\end{align}
Let $\epsilon>0,$ $\zeta_{\epsilon}\in C_{c}^{\infty}(\mathbb{R})$ be a
nonnegative mollifier supported in $\left[  -\epsilon,\epsilon\right]  $ and
set $\tau_{t,\epsilon}=\tau_{t}\ast\zeta_{\epsilon}$. The truncation function
$\tau_{t}$ is bounded by $t,$ and $1$-Lipschitz, $\left\vert \tau_{t,\epsilon
}\left(  s\right)  \right\vert \leqslant t.$ By properties of convolution,
\[
\left\vert \tau_{t,\epsilon}\left(  s\right)  \right\vert \leqslant t,\text{
}\left\vert \tau_{t,\epsilon}\left(  s\right)  -\tau_{t}\left(  s\right)
\right\vert \leqslant\epsilon.
\]
Since almost everywhere derivative of $\tau_{t}$ equals $1_{\left(
-t,t\right)  }$,
\[
\tau_{t,\epsilon}^{\prime}\left(  s\right)  =\int1_{\left(  -t,t\right)
}\left(  s-r\right)  \zeta_{\epsilon}\left(  r\right)  dr\in\left[
0,1\right]  .
\]
Moreover, noting that $\left(  \tau_{t}\ast\zeta_{\epsilon}\right)  ^{\prime
}=\tau_{t}^{\prime}\ast\zeta_{\epsilon}$ and using \cite[C.5 Theorem 7
(iv)]{evans2022partial}, we have
\[
\tau_{t,\epsilon}^{\prime}\left(  s\right)  \rightarrow1_{\left(  -t,t\right)
}\left(  s\right)  ,\text{ }\forall s\neq-t,t.
\]
Recall the cutoff function $\chi_{R}$ of Lemma \ref{lem:BV}. By integration by
parts,
\begin{equation}
\int\chi_{R}f\tau_{t,\epsilon}^{\prime}\left(  w\right)  \left\vert \nabla
w\right\vert ^{2}dx=-\int\chi_{R}f\tau_{t,\epsilon}\left(  w\right)  \Delta
wdx-\int\chi_{R}\tau_{t,\epsilon}\left(  w\right)  \nabla w\cdot dDf-\int
f\tau_{t,\epsilon}\left(  w\right)  \nabla\chi_{R}\cdot\nabla wdx.
\label{eq:truncation-cutoff}
\end{equation}
All terms are well-defined and have integrable dominated functions. Indeed, by
Lemma~\ref{lem:BV} $\left(  \ref{eq:laplacian}\right)  $ and that $\left\vert
\nabla w\right\vert \leqslant2L$, $\left\vert \Delta w\right\vert
\leqslant\Delta u+\Delta v$, we have
\begin{equation}
\int\left\vert \Delta w\right\vert fdx\leqslant\int\left(  \Delta u+\Delta
v\right)  fdx\leqslant2LB_{f}. \label{eq:laplace-wf-bound}
\end{equation}
Also
\begin{equation}
\left\vert \chi_{R}f\tau_{t,\epsilon}^{\prime}\left(  w\right)  \left\vert
\nabla w\right\vert ^{2}\right\vert \leqslant4L^{2}f,\text{ }\left\vert
\chi_{R}\tau_{t,\epsilon}\left(  w\right)  \nabla w\right\vert \leqslant2Lt.
\label{eq:gradient-w-bound}
\end{equation}
Moreover
\[
\left\vert \int f\tau_{t,\epsilon}\left(  w\right)  \nabla\chi_{R}\cdot\nabla
wdx\right\vert \leqslant\frac{2CLt}{R}\Vert f\Vert_{1}\rightarrow0.
\]
For fixed $t>0,$ $\epsilon>0$, Letting $R\rightarrow\infty$ in $\left(
\ref{eq:truncation-cutoff}\right)  $ therefore yields
\begin{equation}
\int\chi_{R}f\tau_{t,\epsilon}^{\prime}\left(  w\right)  \left\vert \nabla
w\right\vert ^{2}dx=-\int f\tau_{t,\epsilon}\left(  w\right)  \Delta
wdx-\int\tau_{t,\epsilon}\left(  w\right)  \nabla w\cdot dDf.
\label{eq:after-R-infinity}
\end{equation}
Since $w$ is smooth, the set
\[
E_{w}=\{w=-t,\nabla w\neq0\}\cup\{w=t,\nabla w\neq0\}
\]
has zero Lebesgue measure. Indeed, since $w$ is smooth, a point $x_{0}\in
E_{w}$ will allow us to use the implicit function theorem to conclude that in
a neighbourhood $U\left(  x_{0}\right)  ,$ all points in $U\left(
x_{0}\right)  \cap E_{w}$ concentrates on a $d-1$ dimensional graph, hence
collectively have zero Lebesgue measure. But $E_{w}$ can be covered by
countably many such neighbourhoods, so $E_{w}$ has at most zero Lebesgue
measure. Therefore, as $\epsilon\rightarrow0,$
\[
\tau_{t,\epsilon}^{\prime}\left(  w\right)  \left\vert \nabla w\right\vert
^{2}\rightarrow1_{\left\{  \left\vert w\right\vert <t\right\}  }\left\vert
\nabla w\right\vert ^{2}=\nabla\tau_{t}\left(  w\right)  \cdot\nabla w,\text{
a.e.}
\]
The equality uses the fact that $\tau_{t}\circ w$ is locally Lipschitz. If
$\epsilon\rightarrow0,$ the sum on the right side of $\left(
\ref{eq:after-R-infinity}\right)  $ converges to
\[
-\int f\tau_{t}\left(  w\right)  \Delta wdx-\int\tau_{t}\left(  w\right)
\nabla w\cdot dDf,
\]
by the dominated convergence theorem with domination provided by $\left(
\ref{eq:gradient-w-bound}\right)  \left(  \ref{eq:gradient-w-bound}\right)  $.
Hence all equations holds in $\left(  \ref{eq:low-level}\right)  $.

\textbf{3}. For $1\leqslant q<\infty.$ Let $A=\int\left\vert w\right\vert
^{q}fdx<\infty$. If $A=\infty,$ there is nothing to prove. Hence we assume
$A<\infty.$ Using $\left(  \ref{eq:low-level}\right)  $ and that
$t^{q}1_{\left\{  \left\vert w\right\vert \geqslant t\right\}  }
\leqslant\left\vert w\right\vert ^{q}$, for all $t>0,$
\begin{equation}
E_{f}=\int_{\left\{  \left\vert w\right\vert <t\right\}  }\left\vert \nabla
w\right\vert ^{2}fdx+\int_{\left\{  \left\vert w\right\vert \geqslant
t\right\}  }|\nabla w|^{2}fdx\leqslant4LB_{f}t+4L^{2}\int_{\left\{  \left\vert
w\right\vert \geqslant t\right\}  }fdx\leqslant4LB_{f}t+4L^{2}t^{-q}A.
\label{eq:threshold}
\end{equation}
At this point we want to justify that $B_{f}>0$ (Recall we are dealing with
nonzero $f$). If $Df=0$, then $\nabla\left(  f\ast\zeta_{\epsilon}\right)
=Df\ast\zeta_{\epsilon}=0$ for all $\epsilon>0$ and $f\ast\zeta_{\epsilon}$ is
constant on $\mathbb{R}^{d}$. This constant is zero because $f\ast
\zeta_{\epsilon}\in L^{1}\left(  \mathbb{R}^{d}\right)  $. Convergence to $f$
in $L^{1}$ would then imply $f=0$, a contradiction. For $A>0,$ we can balance
the powers in the rightmost of $\left(  \ref{eq:threshold}\right)  $ by
choosing $t=(LA/B_{f})^{1/(q+1)}$ and obtain
\[
4LB_{f}t=4LB_{f}L^{1/(q+1)}A^{1/(q+1)}B_{f}^{-1/(q+1)}=4L^{(q+2)/(q+1)}
B_{f}^{q/(q+1)}A^{1/(q+1)},
\]
\[
4L^{2}t^{-q}A=4L^{2}L^{-q/(q+1)}A^{-q/(q+1)}B_{f}^{q/(q+1)}A=4L^{(q+2)/(q+1)}
B_{f}^{q/(q+1)}A^{1/(q+1)}.
\]
The sum of the two terms and taking infimum over $c\in\mathbb{R}$ then yield
$\left(  \ref{eq:interpolation}\right)  .$ If $A=0$, $\left(
\ref{eq:threshold}\right)  $ becomes $E_{f}\leqslant4LB_{f}t$ for all $t>0$.
Letting $t\rightarrow0$ yields $E_{f}=0,$ so the same estimate $\left(
\ref{eq:interpolation}\right)  $ holds.

\textbf{4}. We now come to the nonsmooth $u,v$. Let $u_{\epsilon}=u\ast
\zeta_{\epsilon}$, $v_{\epsilon}=v\ast\zeta_{\epsilon}$, $w_{\epsilon
}=u_{\epsilon}-v_{\epsilon}-c$, where $\zeta_{\epsilon}$ is the same mollifier
as in Lemma~\ref{lem:BV}. These smoothed $u_{\epsilon},v_{\epsilon}$ remain
convex $L$-Lipschitz and the following estimates hold,
\[
\left\Vert w_{\epsilon}-w\right\Vert _{\infty}\leqslant2L\epsilon\text{ and
}\left\vert \left\Vert w_{\epsilon}\right\Vert _{L^{q}(\rho)}-\left\Vert
w\right\Vert _{L^{q}(\rho)}\right\vert \leqslant2L\epsilon\left\Vert
f\right\Vert _{1}^{1/q}.
\]
We also have
\[
\left\vert \nabla u_{\epsilon}-\nabla v_{\epsilon}\right\vert ^{2}
f\leqslant4L^{2}f\in L^{1},
\]
and for almost all $x$,
\[
\left\vert \nabla u_{\epsilon}-\nabla v_{\epsilon}\right\vert ^{2}
\rightarrow\left\vert \nabla u-\nabla v\right\vert ^{2}.
\]
Hence
\[
A_{\epsilon}=\int\left\vert w_{\epsilon}\right\vert ^{q}f\rightarrow A\text{
and }E_{f}=\lim_{\epsilon\rightarrow0}E_{f,\epsilon}=\int\left\vert \nabla
u_{\epsilon}-\nabla v_{\epsilon}\right\vert ^{2}fdx.
\]
For fixed $t>0$, applying $\left(  \ref{eq:threshold}\right)  $ to the smooth
case and sending $\epsilon\rightarrow0$, we get
\[
E_{f}=\lim_{\epsilon\rightarrow0}E_{f,\epsilon}\leqslant\lim_{\epsilon
\rightarrow0}(4LB_{f}t+4L^{2}t^{-q}A_{\epsilon})=4LB_{f}t+4L^{2}t^{-q}A.
\]
Therefore the same argument as before will give us $\left(
\ref{eq:interpolation}\right)  .$

\textbf{5}. For the $q=\infty$ case, fix $c$ with $A=\left\Vert
u-v-c\right\Vert _{L^{\infty}\left(  \rho\right)  }<\infty$. For any $\eta>0$
set $t=A+\eta>0$ and $0<\epsilon<\eta/(2L)$. Then $\left\vert w_{\epsilon
}\right\vert \leqslant A+2L\epsilon<t,$ $\rho$-a.e. Thus $\left(
\ref{eq:low-level}\right)  $ for the smooth case gives
\[
E_{f,\epsilon}=\int_{\left\{  \left\vert w_{\epsilon}\right\vert <t\right\}
}\left\vert \nabla w_{\epsilon}\right\vert ^{2}fdx\leqslant4LB_{f}(A+\eta).
\]
Sending $\epsilon\rightarrow0$, then $\eta\rightarrow0$ yields $E_{f}
\leqslant4LB_{f}A$. Finally minimizing over $c$ we obtain $\left(
\ref{eq:infinity}\right)  $.
\end{proof}

\section{Quantitative compactness of convex functions}

In the following $\operatorname{aff}S$ denote the affine hull of a set
$S\subset\mathbb{R}^{d},$ $\operatorname{ri}S$ denotes the relative interior
of convex set $S$ \cite[page 6, 44]{rockafellar1997convex}.

\begin{lemma}
\label{lem:convex-compactness}Let $C\subset\mathbb{R}^{d}$ be nonempty and
convex, $H=\operatorname{aff}C$, $k=\dim H$, and $D=\operatorname{ri}C$. Write
$B_{H}\left(  x,R\right)  $ for the open ball in $H$ and $\lambda_{H}$ for its
$k$-dimensional volume measure. Suppose $u_{n}:C\mapsto\mathbb{R}$ are finite
convex functions such that for any $\overline{B}_{H}\left(  x,R\right)
\subset D,$ it holds that
\begin{equation}
\sup_{n}\int_{B_{H}\left(  x,R\right)  }\left\vert u_{n}\right\vert
d\lambda_{H}<\infty. \label{eq:local-int}
\end{equation}
Then there exists some subsequence $u_{n_{j}}$ converging uniformly on compact
subsets of $D$ to a finite convex $u$. At each point where $u$ and all
$u_{n_{j}}$ are differentiable relative to $H$, the gradients taken in the
linear space $H-H$ satisfy
\begin{equation}
\nabla_{H}u_{n_{j}}\left(  x\right)  \mapsto\nabla_{H}u\left(  x\right)  .
\label{eq:relative-gradient}
\end{equation}
In particular $\left(  \ref{eq:relative-gradient}\right)  $ holds $\lambda
_{H}$-almost everywhere on $D$.
\end{lemma}

\begin{proof}
\textbf{1}. If $k=0$, condition $\left(  \ref{eq:local-int}\right)  $ becomes
bounds of the functions at a single point, and all relative gradients are
zero. Assume $k\geqslant1$. Fix $x_{0},$ $R>0,$ $\overline{B}_{H}\left(
x_{0},4R\right)  \subset D$. Let
\[
A_{R}=\sup_{n}\int_{B_{H}(x_{0},4R)}|u_{n}|d\lambda_{H}.
\]
and denote the volumn of $\lambda_{H}$ by $\left\vert \cdot\right\vert ,$
\[
\left\vert B_{H}\left(  x_{0},s\right)  \right\vert \triangleq\lambda
_{H}\left(  B_{H}\left(  x_{0},s\right)  \right)  .
\]
For $x\in\bar{B}_{H}\left(  x_{0},2R\right)  $, by Janen's inequlity we have
\[
u_{n}\left(  x\right)  \leqslant\frac{1}{\left\vert B_{H}\left(  x,R\right)
\right\vert }\int_{B_{H}\left(  x,R\right)  }u_{n}d\lambda_{H}=\frac{A_{R}
}{\left\vert B_{H}\left(  x,R\right)  \right\vert }.
\]
By convexity,
\[
u_{n}\left(  x\right)  \geqslant2u_{n}\left(  \left(  x+y\right)  /2\right)
-u_{n}\left(  y\right)  .
\]
For $x\in\bar{B}_{H}\left(  x_{0},2R\right)  ,$ integtating against $y$ over
$B_{H}\left(  x,R\right)  $, we get
\begin{align*}
u_{n}\left(  x\right)   &  \geqslant\frac{2}{\left\vert B_{H}\left(
x,R\right)  \right\vert }\int_{B_{H}\left(  x,R\right)  }u_{n}\left(  \left(
x+y\right)  /2\right)  d\lambda_{H}\left(  y\right)  -\frac{1}{\left\vert
B_{H}\left(  x,R\right)  \right\vert }\int_{B_{H}\left(  x,R\right)  }
u_{n}d\lambda_{H}\\
&  \geqslant\frac{2}{\left\vert B_{H}\left(  x,R/2\right)  \right\vert }
\int_{B_{H}\left(  x,R/2\right)  }u_{n}\left(  y^{\prime}\right)  d\lambda
_{H}\left(  y^{\prime}\right)  -\frac{1}{\left\vert B_{H}\left(  x,R\right)
\right\vert }\int_{B_{H}\left(  x,R\right)  }u_{n}d\lambda_{H}\\
&  \geqslant-2A_{R}/\left\vert B_{H}\left(  x,R/2\right)  \right\vert
-A_{R}/\left\vert B_{H}\left(  x,R\right)  \right\vert .
\end{align*}
The second inequality uses that averaging of a convex function over a ball is
a non-decreasing function of the radius. Therefore
\[
\left\vert u_{n}\left(  x\right)  \right\vert \leqslant2A_{R}/\left\vert
B_{H}\left(  x,R/2\right)  \right\vert +A_{R}/\left\vert B_{H}\left(
x,R\right)  \right\vert \triangleq C_{R},\text{ }\forall x\in\bar{B}
_{H}\left(  x_{0},2R\right)  .
\]

\textbf{2}. Let $x\in\bar{B}_{H}\left(  x_{0},2R\right)  $, $p\in\partial
_{H}u_{n}$ the relative subgradient and $e$ be a unit vector in the linear
space $H-H$. By the definition of subgradient,
\[
u_{n}\left(  x+Re\right)  \geqslant u_{n}\left(  x\right)  +Rp\cdot e,\text{
}u_{n}\left(  x-Re\right)  \geqslant u_{n}\left(  x\right)  -Rp\cdot e.
\]
Hence
\begin{equation}
-2C_{R}\leqslant u_{n}\left(  x\right)  -u_{n}\left(  x-Re\right)  \leqslant
Rp\cdot e\leqslant u_{n}\left(  x+Re\right)  -u_{n}\left(  x\right)
\leqslant2C_{R}. \label{eq:subgrad}
\end{equation}
Taking the supremum over unit vectors yields
\begin{equation}
\left\vert p\right\vert \leqslant2C_{R}/R. \label{eq:subgradient-bound}
\end{equation}
It follows that, for any $x,y\in\bar{B}_{H}\left(  x_{0},2R\right)  ,$
$p_{x}\in\partial_{H}u_{n}\left(  x\right)  ,$ $p_{y}\in\partial_{H}
u_{n}\left(  y\right)  ,$ we have the subgradient inequalities at $x$ and
$y,$
\[
u_{n}\left(  x\right)  \geqslant u_{n}\left(  y\right)  +\left(  x-y\right)
\cdot p_{y},\text{ }u_{n}\left(  y\right)  \geqslant u_{n}\left(  x\right)
+\left(  y-x\right)  \cdot p_{x}.
\]
Rearranging and using the subgradient bounds $\left(
\ref{eq:subgradient-bound}\right)  $ we have
\[
u_{n}\left(  x\right)  -u_{n}\left(  y\right)  \geqslant\left(  x-y\right)
\cdot p_{y}\geqslant-2C_{R}/R\left\vert x-y\right\vert ,
\]
\[
u_{n}\left(  x\right)  -u_{n}\left(  y\right)  \leqslant\left(  x-y\right)
\cdot p_{x}\leqslant2C_{R}/R\left\vert x-y\right\vert .
\]
So for $x,y\in\bar{B}_{H}\left(  x_{0},2R\right)  ,$
\begin{equation}
\left\vert u_{n}\left(  x\right)  -u_{n}\left(  y\right)  \right\vert
\leqslant(2C_{R}/R)\left\vert x-y\right\vert .
\label{eq:convex-local-Lipschitz}
\end{equation}
Arzel\`{a}-Ascoli's theorem therefore gives a convergent subsequence which
converges pointwise to some finite convex function $u$. Outside a common null
set, $u_{u},$ $n\geqslant1$ and $u$ are differentiable everywhere. Let $x$ be
any point where $u_{u},$ $n\geqslant1$ and $u$ are differentiable, $e_{j}$ the
$j$-th coordinate unit vector of the $k$-dimensional linear space $H-H$. The
inequality $\left(  \ref{eq:subgrad}\right)  $ yields that for any $p_{n}
\in\partial_{H}u_{n}(x)=\left\{  \nabla_{H}u_{n}(x)\right\}  $, where
$\nabla_{H}$ denotes the the relative gradient. Hence for all small $t>0,$
\[
\frac{u_{n}\left(  x\right)  -u_{n}\left(  x-te_{j}\right)  }{t}\leqslant
p_{n}\cdot e_{j}\leqslant\frac{u_{n}\left(  x+te_{j}\right)  -u_{n}\left(
x\right)  }{t}.
\]
Using the local uniform convergence, we obtain by letting $n\rightarrow\infty$
that, for all small $t>0,$
\[
\frac{u\left(  x\right)  -u\left(  x-te_{j}\right)  }{t}\leqslant
\liminf_{n\rightarrow\infty}p_{n}\cdot e_{j}\leqslant\limsup_{n\rightarrow
\infty}p_{n}\cdot e_{j}\leqslant\frac{u\left(  x+te_{j}\right)  -u\left(
x\right)  }{t}.
\]
Now sending $t\rightarrow0$ yields
\[
\lim_{n\rightarrow\infty}p_{n}\cdot e_{j}=\nabla_{H}u(x)\cdot e_{j},\text{
}\forall j=1,...,k.
\]
This proves that $\nabla_{H}u_{n}(x)=p_{n}\rightarrow\nabla_{H}u(x).$
\end{proof}

\section{Algebraic computations}

\begin{lemma}
\label{lem:alg-product}Let $n$ be a nonnegative integer and let $a_{j}
,b_{j}\in\mathbb{R}$ for $1\leqslant j\leqslant n$. Then
\begin{equation}
\prod_{j=1}^{n}a_{j}-\prod_{j=1}^{n}b_{j}=\sum_{j=1}^{n}\left(  a_{j}
-b_{j}\right)  \prod_{k<j}a_{k}\prod_{k>j}b_{k}. \label{eq:alg-product}
\end{equation}
The convention is that empty product equals one, empty sum equal zero.
\end{lemma}

\begin{proof}
For $n\geqslant1,$ define $Q_{j}=\prod_{k\leqslant j}a_{k}\prod_{k>j}b_{k}$,
$j=0,...,n.$ Factoring only the $j$-th term yields
\[
Q_{j}-Q_{j-1}=\left(  \prod_{k<j}a_{k}\right)  a_{j}\left(  \prod_{k>j}
b_{k}\right)  -\left(  \prod_{k<j}a_{k}\right)  b_{k}\left(  \prod_{k>j}
b_{k}\right)  =\left(  \prod_{k<j}a_{k}\right)  \left(  a_{j}-b_{j}\right)
\left(  \prod_{k>j}b_{k}\right)  .
\]
Hence the right side of $\left(  \ref{eq:alg-product}\right)  $ equals
\[
\sum_{j=1}^{n}(Q_{j}-Q_{j-1})=Q_{n}-Q_{0}=\prod_{j=1}^{n}a_{j}-\prod_{j=1}
^{n}b_{j}.
\]
If $n=0$, the left side of $\left(  \ref{eq:alg-product}\right)  $ is $1-1$
and the right side is the empty sum, so both are zero.
\end{proof}